%% file: main.tex
\documentclass{article}

\usepackage{amsmath,amsfonts,amsthm,enumitem}
\usepackage[utf8]{inputenc}
\usepackage{lipsum}
\usepackage{url}
\usepackage{mathpazo}
\usepackage{bigints}
\usepackage[T1]{fontenc}
\usepackage{stackengine}
\usepackage{yfonts}
\usepackage{graphics}
\usepackage[english]{babel}
\usepackage[title]{appendix}
\usepackage[colorlinks=true,linkcolor=blue,citecolor=blue,%
filecolor=blue,urlcolor=blue,pageanchor=true,plainpages=false,%
linktocpage]{hyperref}
\usepackage{graphicx}
\usepackage{amssymb}
\usepackage{subcaption}
\usepackage[small, sc]{titlesec}
\usepackage{mathrsfs}
\usepackage{empheq}
\usepackage{authblk}
\usepackage{color,xcolor}
\usepackage[left=1.7cm, right=1.7cm, top=3cm]{geometry}
\usepackage{comment}

\numberwithin{equation}{section}
\newtheorem{theorem}{Theorem}[section]
\newtheorem{proposition}[theorem]{Proposition}

\newtheorem{lemma}[theorem]{Lemma}
\theoremstyle{definition}
\newtheorem{definition}[theorem]{Definition}
\newtheorem{remark}[theorem]{Remark}

\definecolor{myred}{rgb}{0.9,0,0}

\definecolor{vargreen}{rgb}{0.0, 0.5, 0.0}

\makeatletter
\newcommand{\mylabel}[2]{#2\def\@currentlabel{#2}\label{#1}}
\makeatother

\renewcommand{\rho}{\varrho}
\renewcommand{\theta}{\vartheta}

\usepackage{xcolor}
\hypersetup{
colorlinks,
linkcolor={red!50!black},
citecolor={blue!50!black},
urlcolor={blue!80!black}
}
\input{defs}

\begin{document}
\title{\textsc{Elastic membranes spanning deformable curves}}
\author{\textsc{F.\  Ballarin}$^1$\thanks{\href{mailto:francesco.ballarin@unicatt.it}{
\texttt{francesco.ballarin@unicatt.it}}}\,\,\,$-$\,\,\textsc{G.\ Bevilacqua}$^2$\thanks{\href{mailto:giulia.bevilacqua@dm.unipi.it}{\texttt{giulia.bevilacqua@dm.unipi.it}}}\,\,\,$-$\,\, \textsc{L.\ Lussardi}$^3$\thanks{\href{mailto:luca.lussardi@polito.it}{
\texttt{luca.lussardi@polito.it}}}\,\,\,$-$\,\,\textsc{A.\  Marzocchi}$^1$\thanks{\href{mailto:alfredo.marzocchi@unicatt.it}{
\texttt{alfredo.marzocchi@unicatt.it}}}
\bigskip\\
\normalsize$^1$ Dipartimento di Matematica e Fisica ``N. Tartaglia", Università Cattolica del Sacro Cuore,\\
\normalsize via della Garzetta 48, I-25133 Brescia, Italy.\\
\normalsize$^2$ Dipartimento di Matematica, Università di Pisa, Largo Bruno Pontecorvo 5, I–56127 Pisa, Italy.\\
\normalsize$^3$ DISMA ``G.L.\,Lagrange'', Politecnico di Torino, c.so Duca degli Abruzzi 24, I-10129 Torino, Italy.\\
}
\date{}

\maketitle
\begin{abstract}
\noindent
We perform a variational analysis of an elastic membrane spanning a closed curve which may sustain bending and torsion. First, we deal with parametrized curves and linear elastic membranes proving the existence of equilibria and finding first-order necessary conditions for minimizers computing the first variation. Second, we study a more general case, both for the boundary curve and for the membrane, using the framed curve approach. The infinite dimensional version of the Lagrange multipliers’ method is applied to get the first-order necessary conditions. Finally, a numerical approach is presented and employed in several numerical test cases.
\end{abstract}

\vspace{0.4cm}

\


\vspace{0.4cm}

\section{Introduction}

In this paper, we are interested in the variational analysis of energy functionals of the type
\begin{equation}\label{funzionale_iniziale}
\mathcal{E}[\vect{\gamma}, \vect{X}] = \int_{\vect{\gamma}} f(\kappa, \tau)\, d\ell + \int_D \Psi(\nabla \vect{X})\, dudv,
\end{equation}
where $\gamma$ is a closed curve in $\R^3$ with curvature $\kappa$ and torsion $\tau$, $D$ is the unit disc in $\R^2$ and $\vect X\colon D \to \R^3$ is a parametrization of a membrane {\it spanning $\gamma$.} Precisely, we study minimizers among all $\vect X $ with trace pointwise equal to $\vect{\gamma}$: this choice of spanning condition will be crucial to derive first-order necessary conditions for minimizers. Physically, we are considering minimizers of elastic membranes modeled through their deformation gradient $\nabla \vect{X} \in \R^{3\times2}$. The literature is quite rich: we mention, for instance, tension field theory, which has been introduced to study fundamental stress-strain relations for isotropic elastic membranes \cite{pipkin1986relaxed, belytschko1996selected}, and the rigorous derivation of its form and properties obtained around 1990 by Le Dret and Raoult through a $\Gamma$-convergence result \cite{le1993modele,le1995nonlinear}. 

The study of variational problems involving curves and surfaces is a classical topic: there exists a huge literature on the existence of minimizers, while the study of critical points received less attention. 

The most famous example of a variational problem for surfaces is the minimization of the area functional which dates back to Lagrange \cite{lagrange1760essai} and Plateau \cite{plateau1873experimental}: they formulated, separately from a mathematical and physical point of view respectively, the famous {\em Plateau problem}. The first rigorous solution has been presented around 1930 by Douglas and Rad\'{o} \cite{Douglas, Rado} where the area functional is replaced by the Dirichlet integral in conformal coordinates. From that result, in the last century, many mathematicians looked for a more general and weaker definition of surface, for instance integral currents, finite perimeter sets and Almgren minimal sets (see for instance \cite{federer1960normal,almgren1968existence, taylor1976structure} and a complete review \cite{david2014should} with references therein). We also mention recent results by Lytchak and Wenger \cite{lytchak2017area} and by Creutz \cite{creutz2019plateau} where the existence of a minimizer for the classical area functional for disc-type surfaces spanning a fixed curve with possible self-intersections is investigated. 

If the boundary curve is assumed to be elastic, 
few results are known, both in the existence of solutions and in the derivation of Euler-Lagrange equations. 
In \cite{bernatzki2001minimal}, Bernatzki and Ye considered Euler elastic curves ($f=f(\kappa)=\kappa^2$) spanned by area minimizing disc-type surfaces and they proved existence of minimal Jordan curves. We point out that to avoid self-intersections they were forced to introduce a technical condition on the surface tension coefficient (see Hypothesis 2.1 in \cite{bernatzki2001minimal}). Next, Giomi and Mahadevan studied the stability of critical points of the elastic Plateau problem using asymptotic techniques and provided numerical simulations \cite{giomi2012minimal}. We also mention \cite{palmer2021minimal} where minimal surfaces spanning an elastic ribbon described through the Darboux framework are considered. Another possibility to avoid self-intersections is to consider an elastic rod instead of an elastic curve. In this direction there are some results which take into account more general dependence on $\kappa$ and $\tau$. Precisely, existence of minimizers under physical constraints can be found in \cite{biria2014buckling, biria2015theoretical, giusteri2017solution, bevilacqua2018soap, bevilacqua2019soap, bevilacqua2020dimensional}, whereas the derivation of the Euler-Lagrange equations for the critical points seems to be hard for two reasons. On one hand, the presence of constraints on the rod makes it difficult to work with the corresponding variations. On the other hand, the regularity of the contact set between rod and surface is not yet known even for a fixed simple boundary \cite{de2017direct}. 

In this paper we aim to allow for self-intersections, which are physically evident, and to treat a general dependence on $\kappa,\tau$. A first result in this direction is our contribution \cite{bevilacqua2021variational} where we derived the existence of minimizers of a functional of the following type
\[
\mathcal{F}[\vect{\gamma}]= \int_{\vect{\gamma}}f(\kappa, \tau)\, d\ell,
\]
where $\kappa$ and $\tau$ are suitable notions of {\em signed curvature} and {\em torsion} of the curve $\vect{\gamma}$ which may not be simple. Moreover, in the same paper, we also investigated the first-order necessary conditions for minimizers by means of an application of the infinite dimensional version of the Lagrange multipliers' method. This was one of the first results where the elastic curve is less regular, at most $C^1$, and the energy density penalizes not only the curvature but also the torsion (see for classical theory of elasticae \cite{dall2017minimal,garcke2020long} and references therein, and also the papers by Freddi {\em et al.} concerning the Sadowsky functional \cite{freddi2016corrected, freddi2016variational, freddi2022stability}).


Arising from the above considerations, the next natural step should be to add to $\mathcal F$ the area of the spanning surface. This problem leads to several difficulties. It seems that one should use the framework of Lytchak and Wenger \cite{lytchak2017area} and Creutz \cite{creutz2019plateau} in the Sobolev setting while Geometric Measure Theory is needed to deal with general surfaces. 

For all these reasons, here as a first result, we decide to consider classical elastic energy for the membrane surface: this leads to the  functional \eqref{funzionale_iniziale}. In the following, we give more details about our strategies. Let us first discuss the right formulation of the line integral on the elastic curve, namely 
\[
\int_{\vect{\gamma}}f(\kappa, \tau)\, d\ell.
\]
We notice that it is not convenient to see $\gamma$ as $\partial D$ like in \cite{giomi2012minimal, palmer2021minimal}: the integral over the curve $\vect{\gamma}$ has to be rewritten as a function of the the trace of $\vect{X}$ and this leads to a much more complicated expression of the curvature since the trace of $\vect{X}$ may not be parametrized by arc-length. Moreover, for the same reason, we would need to require higher regularity on $\vect{X}$, loosing compactness since only its $H^1$ norm can be bounded along minimizing sequences.
A first successfully approach is to use parametrized curves. In this setting, we can treat the Euler-Bernoulli elastica for $\vect{\gamma}$ and linear elastic membranes. More precisely, in Section \ref{sec:first-approach} we consider the following functional
\begin{equation}
\label{eq:E_parametrico}
     \mathcal E_p[\vect{r},\vect{X}] = \int_0^{2 \pi} \abs{\vect{r}''}^2\, d\vartheta + \frac{\mu}{2}\int_D \textgoth{C}\nabla \vect{X}:\nabla\vect{X}\, dudv.
\end{equation}
Here, $\vect{r}\in W^{2,2}((0, 2\pi);\R^3)$ is a $C^1$-periodic curve clamped at a fixed point and parametrized by arc-length. Moreover, $\mu >0$ is the shear modulus, $\vect{X}\in W^{1,2}(D;\R^3)$ is the disc-type parametrization of the membrane, $\textgoth{C}$ is the fourth order linear elasticity tensor which should depend on the point $\vect{p} \in D$. In addition, we require that 
\[
\hat{\vect{X}}\circ c(\vartheta)= \vect{r}(\theta), \quad \text{a.e.\,$\theta \in (0,2 \pi)$},
\]
where $\hat{\vect{X}}$ denotes the trace of $\vect{X}$ and $c(\vartheta) = (\cos\vartheta,\sin\vartheta)$. We observe that using the trace condition as in Douglas-Rad\'{o} approach, where $\hat{\vect{X}}$ is a reparametrization of $\vect{\gamma}$, we cannot prove the existence of minimizers since their method works only for a prescribed Jordan curve. Nevertheless, the pointwise constraint easily permits to apply the Direct Methods of the Calculus of Variations since the trace properties of Sobolev functions hold true. On the other hand, such a pointwise constraint strongly complicates the derivation of the Euler-Lagrange equations: it is hard, in general, to find variations of the surface which fulfill the trace constraint. For that reason we introduce an auxiliary energy functional where the trace constraint is imposed penalizing the $L^2$ distance between $\hat{\vect{X}}$ and $\vect{\gamma}$ and we derive the Euler-Lagrange equations for the ``free" energy, namely the one with no constraints. Then, we prove that when the associated Lagrange multiplier $\lambda$ goes to $+\infty$, then minimizers converge to the ones of the original functional $\mathcal{E}_p$ defined in \eqref{eq:E_parametrico}. Summing up, the main Theorem of Section \ref{sec:first-approach} is the next one.
\begin{theorem}{\bf(Parametrized curve approach, see Theorems \ref{th:minimo}, \ref{thm:vario_sup_no_curva} and \ref{thm:system_tot})}\\
\noindent
The functional $\mathcal{E}_p$ admits minimizers and they satisfy the following first-order necessary conditions
\[
\left\{
\begin{aligned}
&{\rm div}(\textgoth{C}\nabla \vect{X}) = 0, \quad \text{on $D$},\\
&2\vect r'''+2|\vect r''|^2\vect r'+\mu \vect F_\perp = 0, \quad \text{on $[0,2\pi]$}
\end{aligned}
\right.
\]
where 
\[
\vect F(\theta)=\int_0^\theta(\textgoth{C}\nabla \vect{X} \circ c)\vect{\nu}_D\,d\theta,
\]
$\vect{\nu}_D$ is the outer normal of the disc $D$ and the subscript $\perp$ stands for the orthogonal part with respect to $\vect r'$.
\end{theorem}
We observe that using the parametric approach we are not able to get a full variational analysis if we consider non-linear surface energies of the form 
\[
\int_D \Psi(\nabla \vect{X})\, dudv.
\]
For more details, we refer to Remark \ref{rem_elastica}.

In order to deal with more general energies, both for the boundary curve and for the membrane, in Section \ref{sec:second-approach} we use the framed curve approach introduced by Bishop \cite{bishop}. We introduce a moving orthonormal frame $\left\{\vect{t},\vect{n},\vect{b}\right\} \in W^{1, p}((0, 2 \pi); SO(3))$ with $p >1$ which generates a curve $\vect{r}$ by integration and we define the energy in terms of the frame as
\begin{equation}
\label{eq:E_framed}
    \mathcal{E}_{f}\left[\left(\vect{t}|\vect{n}|\vect{b}\right), \vect{X} \right]= \int_0^{2 \pi} f(\kappa,\tau)\, d\vartheta + \frac{\mu}{2} \int_D \Psi(\nabla \vect{X})\, dudv,
\end{equation}
where $\kappa = \vect{t}'\cdot \vect{n}$ and $\tau= \vect{n}'\cdot \vect{b}$ play the role of weak signed curvature and weak torsion respectively. Here, as in the previous case, the pointwise trace constraint holds, i.e.
$$
\hat{\vect{X}}\circ c(\theta)= \vect{r}(\theta), \quad \text{a.e.\,$\theta \in (0,2 \pi)$}.
$$
The main Theorem of Section \ref{sec:second-approach} is the following one.

\begin{theorem}{\bf(Framed curve approach, see Theorems \ref{th:esistenza_frame} and \ref{th:lagrange_multiplier_nostro})}\\
\noindent
Assume $f$ is a convex function and $f(a,b) \geq c_1 \abs{a}^p+ c_2 \abs{b}^p + c_3$ for some $c_1, c_2 >0$ and $c_3 \in \R$ and $p >1$. Moreover, assume $\Psi$ is continuous, quasiconvex (see Definition \ref{def:quasiconvex}) and satisfying $c_4 \abs{\tens{A}}^q \leq \Psi(\tens{A})\leq c_5 \abs{\tens{A}}^q$ for some $c_4, c_5 >0$ and $q >1$. Then $\mathcal{E}_f$ has a minimizer. Moreover, if 
\begin{itemize}
    \item[ (a) ] $f$ is of class of $C^1$ and that there is $\alpha>0$ such that 
\begin{equation*}
    \begin{aligned}
    &\abs{f_a(a,b)} \leq \alpha\left(1 + \abs{a}^{p-1} +\abs{b}^{p-1}\right),&&&\abs{f_b(a,b)} \leq \alpha\left(1 + \abs{a}^{p-1} +\abs{b}^{p-1}\right),
    \end{aligned}
\end{equation*}
(here $f_a,f_b$ are the partial derivatives
of $f$ with respect to $a$ and $b$ respectively);
\item[ (b) ] there is $\beta>0$ such that 
\[
\abs{{\rm D}\Psi(\tens{A})} \leq \beta \left(1 + \abs{\tens{A}}^2\right)
\]
for all $\tens{A} \in \R^{3\times 2}$,
\end{itemize}
then for each minimizer there exist $\omega \in L^p((0, 2 \pi))$, $\vect{\chi} \in L^2((0,2 \pi);\R^3)$ and $\vect{\lambda} \in \R^3$ such that the following first-order necessary conditions hold a.e. $(u,v) \in D$ and $\vartheta \in [0, 2 \pi]$
\[
\left\{
\begin{aligned}
&\diver\,\left({\rm D}\Psi(\nabla \vect{X})\right) = 0, &&\hbox{on }D,\\
 &-f_a' -\omega \tau -\vect{n} \cdot \int_0^\vartheta \vect{\chi} \, ds  = \vect{\lambda}\cdot \vect{n}, &&\hbox{on }\partial D,\\
 & -f_a \tau +\omega' + f_b \kappa - \vect{b} \cdot\int_0^\vartheta \vect{\chi}\, ds  = \vect{\lambda}\cdot \vect{b}, &&\hbox{on }\partial D,\\
	&\frac{\mu}{2} \left({\rm D}\Psi(\nabla \vect{X} \circ c)\right)\vect{\nu}_D = \vect{\chi}, &&\hbox{on }\partial D,\\
 & \omega \kappa -f'_b = 0,&&\hbox{on }\partial D.
\end{aligned}
\right.
\]
\end{theorem}
\medskip

Subsequently, we are able to eliminate all the Lagrange multipliers for smooth critical points whenever $\kappa \neq 0$ (see Theorem \ref{th:elimino_l_m}) and in the case where $\Psi(A) =|A|^2$ and $f(\kappa) = \kappa^2$, we show that the two approaches coincide (see Remark \ref{rem:uguali}). The power of the framed curve approach stems also in the fact that to get first-order necessary conditions for minimizers we do not have to find suitable variations but we can directly apply the infinite dimensional Lagrange multipliers' Theorem. 

Finally, in Section \ref{sec:third-approach} we introduce an {\em ad hoc} method to be employed in the numerical simulations and we test it in some cases in Section \ref{sec:numerical-results}. We propose a suitable discretization of the functional
\[
[\vect{\gamma}, \vect{X}] \mapsto  \int_{\vect{\gamma}} \kappa^2 \, d\ell + \int_D \textgoth{C}\nabla \vect{X} : \nabla \vect{X} \, dudv.
\]
Here, we deem convenient to identify the elastic curve $\vect{\gamma}$ with the trace of the parametrized map $\vect{X}$. On one hand, this choice is computationally convenient since it allows to consider a single unknown $\vect{X}$, thus satisfying by construction the identification constraints between $\vect{X}$ and $\vect{\gamma}$. On the other hand, we need to increase the necessary regularity $\vect{X} \in W^{3,2}(D; \R^3)$ for the first integral to be well-defined, namely to guarantee that $\hat{\vect{X}} \in W^{\frac{5}{2}, 2}(\partial D;\R^3)$. The corresponding discretization by a finite element method would thus require a $C^2$ continuity across cells, which is different from the standard $C^0$ requirement of Lagrange finite elements. Since the definition of finite elements with such an higher order continuity is challenging and not readily available on commonly used open-source implementations, in Section \ref{sec:numerical-results} we will exploit an alternative way to reformulate the discrete problem. Furthermore, another important aspect of the numerical approach is how to impose the length preserving constraint, since a pointwise formulation is ill-suited in the application of a finite element method. Hence, by introducing an additional unknown, we will replace that constraint by an integral formulation through an infinite dimensional Lagrange multiplier defined only on the boundary of the disc. The additional unknown, being defined only on the boundary, only marginally increases the overall computational cost. Precisely, we introduce the following functional
\begin{equation}
    \label{eq:E_numerico}
    \mathcal{E}_n[\vect{X}] = \int_{\partial D} \abs{\partial^2_{\theta} \hat{\vect{X}}}^2\, d\theta + \frac{\mu}{2}\int_D \textgoth{C}\nabla \vect{X}:\nabla\vect{X}\, dudv,
\end{equation}
where $\vect{X} \in W^{3,2}(D; \R^3)$ is the disc parametrization of the elastic membrane, $\partial_{\theta}$ is the tangential directional derivative on the boundary of the disk $\partial D$. We identify the curve $\vect{r}$ with the trace of the map $\vect{X}$, namely $\hat{\vect{X}} \in W^{\frac{5}{2}, 2}(\partial D;\R^3)$. Furthermore, we require 
\[
\left|\partial_\theta \hat{\vect{X}}\right| = 1,\quad\text{on}\quad\partial D,
\qquad \text{and} \qquad
\int_D \vect{X} \; du dv = \vect{0},
\]
i.e., we are imposing the length preserving constraint and avoiding rigid motions. 
However, in practice for technical convenience the first constraint is imposed weakly by means of the following auxiliary functional
\[
    \widetilde{\mathcal{E}}_n[\vect{X}, \ell]= \int_{\partial D} \abs{\partial^2_{\theta} \hat{\vect{X}}}^2\, d\vartheta + \frac{\mu}{2}\int_D \textgoth{C}\nabla \vect{X}:\nabla\vect{X}\, dudv + \int_{\partial D} \ell \; \left[\left|\partial_\theta \hat{\vect{X}}\right|^2 - 1\right] d\vartheta,
\]
where $\ell \in W^{\frac{1}{2},2}(\partial D;\R)$ is an infinite dimensional Lagrange multiplier. In this setting, we are able to prove the following Theorem.
\begin{theorem}{\bf(Numerical approach, see Equation \eqref{eq:system_tot1})}\\
\noindent 
The following first-order necessary conditions hold

\[
\left\{
\begin{aligned}
&\mu \int_D \textgoth{C}\nabla \vect{X}: \nabla \vect{V}\, dudv
+ 2 \int_{\partial D} \partial^2_{\theta} \hat{\vect{V}} \cdot \partial^2_{\theta} \hat{\vect{X}}\, d\theta
+ 2 \int_{\partial D} \ell \; \partial_\theta \hat{\vect{V}} \cdot \partial_\theta \hat{\vect{X}} \; d\theta = 0,\\
&\int_{\partial D} m \; \left[\left|\partial_\theta \hat{\vect{X}}\right|^2 - 1\right] d\theta = 0,
\end{aligned}
\right.
\]
for every $(\vect{V}, m) \in C^\infty(\overline{D}; \R^3) \times C^\infty(\overline{\partial D}; \R)$ such that $\int_D \vect{V} \; du dv = \vect{0}$.
\end{theorem}

This numerical approach will then be employed in several numerical test cases which will show a typical behavior of the Dirichlet energy functional, i.e.\,the possibility of planar or non-planar self-intersections \cite{morgan2016geometric}, and in the future it can be used in biological researches. Indeed, the phenomenon of self-intersections can be interpreted as the physical absorption of knotted proteins by cellular membranes \cite{evans2018mechanics}. However, the proposed numerical approach can only penalize the curvature of the curve and consider linear elastic membrane. Future efforts must be devoted in developing a numerical method capable of handling the dependence of the energy density function $f$ on the torsion plus considering a non-linear elastic energy density $\Psi$, as introduced in the framed curve approach.

\section{The parametrized curve approach}
\label{sec:first-approach}
\subsection{Setting of the problem}
We assume that the boundary elastic curve is closed 
and is clamped at a fixed point $\vect x_0$ and it has fixed length. More precisely, fix $\vect{x}_0 \in \R^3$ and let $\vect{r} \in W^{2,2}((0,2\pi);\R^3)$ be such that:
\begin{align}
&\label{eq:r_0} \vect r(0)=\vect r(2\pi)=\vect x_0,\\
&\label{eq:rprimo0}\vect r'(0) =\vect r'(2\pi),\\
&\label{eq:velocita_1}\left|\vect{r}'\right| = 1, \quad \text{on $[0,2 \pi]$}.
\end{align}
Notice that conditions \eqref{eq:r_0} and \eqref{eq:rprimo0} make sense since $\vect r\in C^1([0,2\pi];\R^3)$. The elastic membrane is assumed to be a smooth disc-type parametrization $\vect{X}\in W^{1,2}(D; \R^3)$, where $D := \{(u,v) \in \R^2 : u^2 + v^2 < 1\}$. Denoting by $\hat{\vect X}$ the trace of $\vect X$ on $\partial D$ and by $c(\theta)=(\cos\theta,\sin\theta)$, the following constraint has to be satisfied by $\vect r$ and $\vect{X}$
\begin{align}
     \label{eq:traccia}
\hat{\vect X}\circ c = \vect r, \quad \text{a.e.\,on $(0,2 \pi)$}.
\end{align}
In the following, with the notation $c^{-1}$ we will mean the inverse of $c$ restricted on $[0,2\pi)$. Notice that the quantity $\hat{\vect X}\circ c$ is well defined, at least a.e.\,on $(0,2\pi)$, even if $\hat{\vect X} \in W^{\frac{1}{2},2}(\partial D;\R^3)$, since the map $c$ is a bijection as a map $[0,2\pi) \to \partial D$. Let
\[
\mathcal W_p=W^{2,2}((0,2 \pi); \R^3) \times W^{1,2}(D; \R^3),
\]
where the subscript $p$ refers to the {\em parametric approach} adopted here. On $\mathcal W_p$ we consider the product weak topology, then $(\vect r_h,\vect X_h) \rightharpoonup (\vect r,\vect X)$ means that $\vect r_h\rightharpoonup \vect r$ in $W^{2,2}((0,2 \pi); \R^3)$ and $\vect X_h\rightharpoonup \vect X$ in $W^{1,2}(D;\R^3)$. The set of constraints is given by
\[
\mathcal C_p= \left\{(\vect{r}, \vect{X})\in \mathcal W_p: \text{\eqref{eq:r_0} - \eqref{eq:rprimo0} - \eqref{eq:velocita_1} - \eqref{eq:traccia} hold true}\right\}.
\]
We recall that the scalar product between two matrices $A=(a_{ij}),B=(b_{ij})$ is defined by
\[
A:B=\sum_{i=1}^3\sum_{j = 1}^2a_{ij}b_{ij}.
\]
Let us now introduce the energy functional. Let $\mathcal E_p \colon \mathcal C_p\to [0,+\infty)$ be given by
\[
    \mathcal E_p[\vect{r},\vect{X}] = \int_0^{2 \pi} \abs{\vect{r}''}^2\, d\vartheta + \frac{\mu}{2}\int_D \textgoth{C}\nabla \vect{X}:\nabla\vect{X}\, dudv
\]
where $\mu>0$ is the shear modulus, measuring the softness of the membrane, and $\textgoth C \in H^1(D;\mathcal L(\R^{3\times 2};\R^{3\times 2}))$ with
\begin{equation}\label{coerc}
\alpha\abs{\zeta}^2 \le \textgoth{C}(u,v)\zeta : \zeta \leq \beta \left(1+ \abs{\zeta}^2\right), \qquad \forall\,(u,v)\in D,\, \forall \zeta \in \R^{3\times2},
\end{equation}
for some $\alpha,\beta >0$ independent of $(u,v) \in D$.

Our goal is to minimize $\mathcal E_p$ and to find first-order necessary conditions for suitable minimizers.

\subsection{Existence of minimizers of \texorpdfstring{${\mathcal E}_p$}{E}}

In this subsection, we are going to prove closure of the constraint set $\mathcal{C}_p$ and the compactness of sequences of equi-boundeness energy. As a consequence, existence of minimizers follows.

\begin{lemma}\label{compactness}
Let $(\vect{r}_h, \vect{X}_h) \in \mathcal{C}_p$ be such that $\mathcal E_p[\vect{r}_h, \vect{X}_h]\le c$ for some $c>0$. 
Then there exists $(\vect{r}, \vect{X})\in \mathcal C_p$ such that, up to a subsequence, $(\vect{r}_h, \vect{X}_h) \stackrel{}{\rightharpoonup } (\vect{r}, \vect{X})$ weakly in $\mathcal W_p$. 
\end{lemma}

\begin{proof}
First of all, we have $\norm{\vect{r}''_h}_{L^2} \leq c$, hence \eqref{eq:velocita_1} implies that $\|\vect r'_h\|_{L^2}=\sqrt{2\pi}$. Since $\vect r_h$ is clamped at $\vect x_0$, we get 
\[
\vect r_h(\theta)=\vect x_0+\int_0^{\theta}\vect r_h'\,ds,
\]
from which $\|\vect r_h\|_{L^2} \le c'$ for some other positive constant $c'$. As a consequence, $\norm{\vect{r}_h}_{W^{2,2}}$ is bounded, and then, up to a subsequence, $\vect{r}_h \rightharpoonup \vect{r}$ in $W^{2,2}((0,2\pi);\R^3)$ for some $\vect{r} \in W^{2,2}((0,2\pi);\R^3)$. By Sobolev embeddings, we get $\vect{r}_h\to \vect{r}$ strongly in $C^1$. Thus, conditions \eqref{eq:r_0} - \eqref{eq:rprimo0} - \eqref{eq:velocita_1} pass easily to the limit. Now, using the coercivity condition \eqref{coerc} for the membrane, we easily deduce that $\|\nabla \vect X_h\|_{L^2}$ is bounded. Now, we prove that also $\|\vect{X}_h\|_{L^2}$ is bounded. For any $(u,v) \in \partial D$ we let $\tilde{\vect {r}}_h(u,v)=\vect{r}_h(c^{-1}(u,v))$. Take $\vect{Y}_h=\text{Ext}\,\tilde{\vect{r}}_h$, where $\text{Ext}: [0,2 \pi]\to D$ is the extension operator solution of the Dirichlet problem with boundary condition $\tilde{\vect{r}}_h$, and consider $\vect{Z}_h=\vect X_h-\vect{Y}_h$. By construction $\vect {Z}_h\in W^{1,2}_0(D;\R^3)$. Combining the triangle inequality with the Poincar\'e inequality and the extension embedding Ext, we get
\[
\begin{aligned}
\|\vect X_h\|_{L^2} &\le \|\vect Z_h\|_{L^2}+\|\vect Y_h\|_{L^2}\\
&\le c_1\|\vect \nabla Z_h\|_{L^2}+\|\vect Y_h\|_{L^2} \\
&\le c_1\|\nabla \vect X_h\|_{L^2} +c_2\|\vect Y_h\|_{W^{1,2}} \\
&\le c_1\|\nabla \vect X_h\|_{L^2} +c_3\|\vect r_h\|_{W^{2,2}}\\
&\le c_4.
\end{aligned}
\]
As a consequence, up to a subsequence, $\vect{X}_h \rightharpoonup \vect{X}$ in $W^{1,2}(D;\R^3)$ for some $\vect{X} \in W^{1,2}(D;\R^3)$. It remains to prove that $(\vect{r}, \vect{X})$ satisfies also condition \eqref{eq:traccia}. Take a continuous function $\psi \colon [0,2\pi] \to \R^3$ with $\psi(0)=\psi(2\pi)$. Let $\varphi \colon \partial D \to\R$ be given by $\varphi(u,v)=\psi(c^{-1}(u,v))$. Then $\varphi \in L^2(\partial D;\R^3)$. The trace operator $W^{1,2}(D;\R^3) \to W^{\frac{1}{2},2}(\partial D;\R^3)$ is linear and continuous, hence also weakly continuous. This implies that
\[
\lim_{h\to+\infty}\int_{\partial D}\hat{\vect{X}}_h\cdot \varphi\,d\mathcal H^1=\int_{\partial D} \hat{\vect X}\cdot \varphi\,d\mathcal H^1,
\]
namely
\[
\lim_{h\to+\infty}\int_0^{2\pi}(\hat{\vect X}_h \circ c)\cdot \psi\,d\theta=\int_0^{2\pi}(\hat{\vect X}\circ c)\cdot \psi\,d\theta.
\]
Since $\hat{\vect X}_h \circ c=\vect r_h$ we obtain
\[
\lim_{h\to+\infty}\int_0^{2\pi}\vect r_h\cdot \psi\,d\theta=\int_0^{2\pi}(\hat{\vect X}\circ c)\cdot \psi\,d\theta.
\]
On the other hand, $\vect{r}_h\to \vect{r}$ uniformly on $[0,2\pi]$, hence
\[
\lim_{h\to+\infty}\int_0^{2\pi}\vect r_h\cdot \psi\,d\theta=\int_0^{2\pi}\vect r\cdot \psi\,d\theta.
\]
By the arbitrariness of $\psi$ we conclude that $\hat{\vect X} \circ c=\vect r$ for a.e.\,on $(0,2 \pi)$, that is \eqref{eq:traccia} and this concludes the proof. 
\end{proof}

\begin{theorem}
    \label{th:minimo}
    There exists a minimizer of $\mathcal{E}_p$ in $\mathcal{C}_p$.
\end{theorem}
\begin{proof}
    The thesis follows applying Lemma \ref{compactness} and noticing that the functional $\mathcal{E}_p$ is convex, hence lower semicontinuous.
 \end{proof}

\subsection{Variational analysis of \texorpdfstring{${\mathcal E}_p$}{E}}
We are now in position to prove the second main result of this section. In the following we will use the following notation: 
\[
({\rm div}\,\vect Y)^i=\frac{\partial \vect Y^i_1}{\partial u}+\frac{\partial \vect Y^i_2}{\partial v}, \quad i=1,2,3,
\]
and we denote with $\vect{\nu}_D$ the outer normal of the disc $D$.

First of all, we derive the first-order necessary conditions for the membrane, while keeping the boundary curve fixed.

\begin{theorem}\label{thm:vario_sup_no_curva}
Let $(\vect r,\vect X) \in \mathcal \mathcal{C}_p$ be a minimizer of $\mathcal{E}_p$. Then 
\begin{equation}\label{eq:solo_superficie}
{\rm div}(\textgoth{C}\nabla \vect{X}) = {\bf 0}, \quad \text{on $D$}. 
\end{equation} 
\end{theorem}

\begin{proof}
Let $\sigma>0$. For any $t\in (-\sigma,\sigma)$ let 
\[
\vect X_t= \vect X  + t \vect V,
\]
be the variation of the membrane with fixed curve, where $\vect{V} \in C^1(\overline{D};\R^3)$ satisfies the trace constraint, namely
\[
    \hat{\vect V}\circ c  = \vect 0, \quad \text{on $[0, 2\pi]$}.
\]
Then, by minimality 
\[
\frac{{\rm d}}{{\rm d t}} \mathcal E_p[\vect r,\vect X_t]_{\big|_{t = 0}}=0,
\]
which easily reads as 
\[
\int_D \textgoth{C}\nabla \vect X: \nabla \vect{V}\, dudv=0.
\]
Integrating by parts we deduce that 
\[
\int_D {\rm div}\,(\textgoth{C}\nabla \vect{X}) \cdot \vect V\, dudv=0,
\]
which gives the thesis \eqref{eq:solo_superficie} just by using the arbitrariness of the test $\vect V$.
\end{proof}

Now, we want to find necessary conditions for any minimizer of $\mathcal E_p$. In order to do that, first of all we need to construct suitable variations of the boundary curve following the idea of Bernatzki and Ye \cite{bernatzki2001minimal}.

\begin{lemma}\label{lemma:variazione_curva}
	Let $\vect{r}\in W^{2,2}((0,2\pi);\R^3)$ be satisfying conditions \eqref{eq:r_0} - \eqref{eq:rprimo0} - \eqref{eq:velocita_1}. Let $\vect \eta \in C^2([0,2\pi];\R^3)$ be such that:
	\begin{itemize}
	    \item[(i)] $\vect{\eta}(0)=\vect{\eta}(2\pi)$;
	    \item[(ii)] $\vect{\eta}'(0)=\vect{\eta}'(2\pi)$;
	    \item[(iii)] $\vect r'\cdot \vect\eta' = 0$ on $[0, 2\pi]$.
	\end{itemize}
	Then there exist a family of curves $\{\vect r^t :t\in (-\varepsilon,\varepsilon) \} \subset W^{2,2}((0,2 \pi);\R^3)$ such that:
	\begin{itemize}
	    \item[(a)] $\vect{r}^0 = \vect r$ on $[0,2\pi]$;
	    \item[(b)] $\vect{r}^t$ satisfies conditions \eqref{eq:r_0} - \eqref{eq:rprimo0} - \eqref{eq:velocita_1} for any $t\in (-\varepsilon,\varepsilon)$;
	    \item[(c)] it holds
	    \[
	    \frac{\partial}{\partial t}|(\vect r^t)''|^2_{\big|_{t=0}}=2\vect r''\cdot \vect\eta'', \quad \text{a.e.\,on $(0,2\pi)$};
	    \]
	    \item[(d)] for any $\vect v,\vect w\in \R^3$ it holds
	    \[
	    \frac{\partial}{\partial t}|\vect v+t\vect w-\vect r^t|^2_{\big|_{t=0}}=2(\vect v-\vect r)\cdot (\vect w-\vect\eta), \quad \text{on $[0,2\pi]$}.
	    \]
	\end{itemize}
\end{lemma}
\begin{proof}
For any $t,\delta\in \R$ let
\[
W(t,\delta)=\left\{\begin{array}{ll}
\displaystyle \frac{1}{t^2}\int_0^{2\pi}\bigg(\sqrt{(1-t^2 \delta)^2 + t^2 |\vect \eta'|^2}-1\bigg)\,d\theta & \text{if $t\ne 0$}\\
\\
\displaystyle -2\pi\delta+\int_0^{2\pi}|\vect \eta'|^2\,d\theta & \text{if $t=0$}.
\end{array}\right.
\]
It is easy to see that $W \in C^1(\R^2)$ and that $W(0,\delta_0)=0$ where
\[
\delta_0=\frac{1}{4\pi}\int_0^{2\pi}|\vect \eta'|^2\,d\theta.
\]
Moreover,
\[
\frac{\partial W}{\partial \delta}(0,\delta_0)=-2\pi\ne 0.
\]
Using the Implicit function Theorem we get the existence of $\varepsilon>0$ and $\delta \in C^1(-\varepsilon,\varepsilon)$ such that $W\left(t,\delta(t)\right)=0$ for any $t\in (-\varepsilon,\varepsilon)$ and $\delta(0)=\delta_0$. We set
\[
\tilde{\vect r}^t(s)=(1-t^2\delta(t))\vect r(s)+t\vect\eta(s), \quad s\in [0,2\pi].
\]
It is easy to see that the curves $\tilde{\vect r}^t$ belong to $W^{2,2}((0,2\pi);\R^3)$ and satisfy conditions \eqref{eq:r_0} and \eqref{eq:rprimo0}. Moreover, using the facts that $|\vect r'|=1$ and $\vect r'\cdot \vect \eta'=0$ everywhere, we get, by explicit computation,
\[
\begin{aligned}
|(\tilde{\vect r}^t)'(s)|&=\sqrt{(1-t^2 \delta(t))^2 |\vect r'(s)|^2 + t^2 |\vect \eta '(s)|^2 + 2 t (1 - t^2 \delta(t))\vect r'(s)\cdot \vect\eta'(s)}\\
&=\sqrt{(1-t^2 \delta(t))^2 + t^2 |\vect \eta '(s)|^2}.
\end{aligned}
\]
As a consequence, we obtain 
\[
\begin{aligned}
\mathcal{L}(\tilde{\vect r}^t)&=\int_0^{2 \pi}|(\tilde{\vect r}^t)'|\,ds\\
&=\int_0^{2 \pi}\sqrt{(1-t^2 \delta(t))^2 + t^2 |\vect \eta '|^2}\,ds\\
&=t^2W(t,\delta(t))+2\pi\\
&=2\pi.
\end{aligned}
\]
Notice that (if necessary choose a smaller $\varepsilon$), 
\begin{equation}\label{inv}
|(\tilde{\vect r}^t)'(s)|\ne 0.
\end{equation}
For any $(t,s) \in (-\varepsilon,\varepsilon)\times [0,2\pi]$ let 
\[
\beta(t,s)=\int_0^s |(\tilde{\vect r}^t)'|\,d\sigma=\int_0^s \sqrt{(1-t^2\delta(t))^2+t^2|\vect\eta'|^2}\,d\sigma. 
\]
Using \eqref{inv}, we can say that for any $t\in (-\varepsilon,\varepsilon)$ the function $\beta(t,\cdot)$ is invertible on $[0,2\pi]$: let us denote by $\alpha(t,\cdot)$ the inverse of $\beta(t,\cdot)$. Finally, for any $(t,\theta) \in (-\varepsilon,\varepsilon)\times [0,2\pi]$ we set 
\[
\vect r^t(\theta)=\tilde{\vect r}^t(\alpha(t,\theta)).
\]
First of all, by composition we get $\vect r^t \in W^{2,2}((0,2\pi);\R^3)$ for every $t\in (-\varepsilon,\varepsilon)$. Moreover, since $\beta(0,s)=s$ we immediately deduce that $\alpha(0,\theta)=\theta$. Hence
\[
\vect r^0(\theta)=\tilde{\vect r}^0(\alpha(0,\theta))=\tilde{\vect r}^0(\theta)=\vect r(\theta)
\]
which is (a). Next, since $\alpha(t,0)=0$ and $\alpha(t,2\pi)=2\pi$, then
\[
\begin{aligned}
\vect r^t(0)&=\tilde{\vect r}^t(\alpha(t,0))\\
&=\tilde{\vect r}^t(0)\\
&=(1-t^2\delta(t))^2\vect r(0)+t\vect\eta(0)\\
&=(1-t^2\delta(t))^2\vect r(2\pi)+t\vect\eta(2\pi)\\
&=\tilde{\vect r}^t(2\pi)\\
&=\tilde{\vect r}^t(\alpha(t,2\pi))\\
&=\vect r^t(2\pi).
\end{aligned}
\]
Furthermore, to prove $|(r^t)'|=1$, we compute
\[
\begin{aligned}
(\vect r^t)'(\theta)&=\frac{\partial}{\partial \theta}\tilde{\vect r}^t(\alpha(t,\theta))\\
&=(\tilde{\vect r}^t)'(\alpha(t,\theta))\frac{\partial}{\partial \theta}\alpha(t,\theta)\\
&=(\tilde{\vect r}^t)'(\alpha(t,\theta)\left(\frac{\partial}{\partial s}\beta(t,s)_{|_{s=\alpha(t,\theta)}}\right)^{-1}\\
&=(\tilde{\vect r}^t)'(\alpha(t,\theta)|(\tilde{\vect r}^t)'(\alpha(t,\theta)|^{-1},
\end{aligned}
\]
deducing immediately the thesis.
Finally, to conclude the proof of (b), we compute
\[
(\vect r^t)'(\theta)=\frac{(1-t^2\delta(t))\vect r'(\alpha(t,\theta))+t\vect\eta'(\alpha(t,\theta))}{\sqrt{(1-t^2\delta(t))^2+t^2|\vect\eta'(\alpha(t,\theta))|^2}},
\]
obtaining
\[
\begin{aligned}
(\vect r^t)'(0)&=\frac{(1-t^2\delta(t))\vect r'(\alpha(t,0))+t\vect\eta'(\alpha(t,0))}{\sqrt{(1-t^2\delta(t))^2+t^2|\vect\eta'(\alpha(t,0))|^2}}\\
&=\frac{(1-t^2\delta(t))\vect r'(0)+t\vect\eta'(0)}{\sqrt{(1-t^2\delta(t))^2+t^2|\vect\eta'(0)|^2}}\\
&=\frac{(1-t^2\delta(t))\vect r'(2\pi)+t\vect\eta'(2\pi)}{\sqrt{(1-t^2\delta(t))^2+t^2|\vect\eta'(2\pi)|^2}}\\
&=\frac{(1-t^2\delta(t))\vect r'(\alpha(t,2\pi))+t\vect\eta'(\alpha(t,2\pi))}{\sqrt{(1-t^2\delta(t))^2+t^2|\vect\eta'(\alpha(t,2\pi))|^2}}\\
&=(\vect r^t)'(2\pi),
\end{aligned}
\]
which gives (b). We now show (c). First of all, using (a) we have, for a.e.\,$\theta \in (0,2\pi)$, 
\[
\frac{\partial}{\partial t}|(\vect r^t)''|^2_{\big|_{t=0}}=2\vect r''\cdot\frac{\partial^2}{\partial \theta^2}\frac{\partial \vect r^t}{\partial t}_{\big|_{t=0}}.
\]
It is sufficient to prove that 
\begin{equation}\label{final}
\frac{\partial \vect r^t}{\partial t}_{\big|_{t=0}}=\vect\eta.
\end{equation}
By direct computation we have
\[
\begin{aligned}
\frac{\partial \vect r^t}{\partial t}(\theta)_{\big|_{t=0}}&=\frac{\partial }{\partial t}\tilde{\vect r}^t(\alpha(t,\theta))_{\big|_{t=0}}\\
&=\vect\eta(\theta)+\vect r'(\theta)\frac{\partial \alpha}{\partial t}(t,\theta)_{\big|_{t=0}}.
\end{aligned}
\]
Now, since $\alpha(t, \cdot)$ is the inverse of $\beta(t, \cdot)$, using the chain rule, we can differentiate the relation $\alpha(t,\beta(t,s))=s$ with respect to $t$ getting
\[
\begin{aligned}
0=\frac{\partial \alpha}{\partial t}(t,\beta(t,s))_{\big|_{t=0}}&=\frac{\partial \alpha}{\partial t}(t,s)_{\big|_{t=0}}+\frac{\partial \alpha}{\partial \theta}(0,\theta)\frac{\partial \beta}{\partial t}(t,\theta)_{\big|_{t=0}}\\
&=\frac{\partial \alpha}{\partial t}(t,s)_{\big|_{t=0}}.
\end{aligned}
\]
Hence, we immediately deduce \eqref{final} and this ends the proof of (c). It remains to prove (d). Thanks to \eqref{final} we have
\[
\frac{\partial}{\partial t}|\vect{v}+t\vect{w}-\vect r^t|^2_{\big|_{t=0}}=2(\vect{v}-\vect r)\cdot \left(\vect{w}-\frac{\partial \vect r^t}{\partial t}_{\big|_{t=0}}\right)=2(\vect{v}-\vect r)\cdot (\vect{w}-\vect\eta),
\]
which concludes the proof.
\end{proof}

To derive the first-order necessary conditions for the entire system, we introduce an auxiliary functional. Hence, for any $\lambda>0$, let $\mathcal E_p^\lambda \colon \mathcal C'_p \to \R$ be given by
\[
    \mathcal E_p^\lambda\left[\vect{r},\vect{X}\right]= \int_0^{2 \pi} \abs{\vect{r}''}^2\, d\vartheta+ \frac{\mu}{2}\int_D \textgoth{C}\nabla \vect{X}:\nabla \vect{X}\, dudv + \lambda \int_0^{2 \pi} \abs{\hat{ \vect X} \circ c - \vect r}^2\, d\vartheta,
\]
where
\[
    \mathcal C'_p= \left\{(\vect{r}, \vect{X})\in \mathcal W_p: \eqref{eq:r_0} - \eqref{eq:rprimo0} - \eqref{eq:velocita_1} \hbox{ hold}\right\}.
\]
In other words, the constraint \eqref{eq:traccia} is penalized through the $L^2$-distance between $\hat{ \vect X}$ and $\vect r$. 

\begin{proposition}
For every $\lambda >0$, the functional $\mathcal E_p^\lambda$ has a minimizer. 
\end{proposition}

\begin{proof}
Take a minimizing sequence for $\mathcal E_p^\lambda$, that is a sequence $(\vect r_h,\vect X_h)$ in $\mathcal C'_p$ with
\[
\lim_{h\to+\infty}\mathcal E_p^\lambda[\vect r_h,\vect X_h]=\inf_{\mathcal C'_p}\mathcal E_p^\lambda.
\]
It is straightforward to see that $\inf_{\mathcal C'_p}\mathcal E_p^\lambda<+\infty$. Hence, we can suppose that $\mathcal{E}_p^\lambda[\vect{r}_h, \vect{X}_h] \le c$ for some $c\ge 0$. Using a similar argument as in the proof of Lemma \ref{compactness} we deduce that there exists $(\vect{r}, \vect{X})\in \mathcal C'_p$ such that, up to a subsequence, $(\vect{r}_h, \vect{X}_h) \stackrel{}{\rightharpoonup } (\vect{r}, \vect{X})$ weakly in $\mathcal C'_p$. The weak lower semicontinuity of the first two terms in $\mathcal E_p^\lambda$ in $\mathcal{W}_p$ is standard (for the membrane contribution we refer to \cite{acerbi1984semicontinuity}). Concerning the trace term, since the trace operator is linear and continuous, it is also weakly continuous, and then $\hat{\vect{X}}_h\rightharpoonup \hat{\vect X}$ in $L^2(\partial D;\R^3)$. As a consequence, $\hat{\vect{X}}_h\circ c -\vect r_h\rightharpoonup \hat{\vect X} \circ c-\vect r$ in $L^2(\partial D;\R^3)$. By the weak lower semicontinuity of the $L^2$-norm we get 
\begin{equation}\label{lsc-trace}
\int_0^{2 \pi} \abs{\hat{ \vect X} \circ c - \vect r}^2\, d\vartheta \le \liminf_{h\to+\infty} \int_0^{2 \pi} \abs{\hat{ \vect X}_h \circ c - \vect r_h}^2\, d\vartheta.
\end{equation}
The thesis follows just applying the direct method of the Calculus of Variations.
\end{proof}

\begin{proposition}
\label{th:gamma_conv}
Let $(\lambda_h)$ be a positive sequence with $\lambda_h \to +\infty$ and let $(\vect{r}_h, \vect{X}_h) \in \mathcal{C}'_p$ with $\mathcal{E}_p^{\lambda_h}[\vect{r}_h, \vect{X}_h]\leq c$, for some $c>0$. Then, there exists $(\vect{r}, \vect{X})\in \mathcal{C}_p$ such that, up to subsequence, $(\vect{r}_h, \vect{X}_h) \rightharpoonup  (\vect{r}, \vect{X})$ weakly in $\mathcal W_p$. Moreover, the family $\{\mathcal{E}_p^{\lambda_h}\}_{h\in \N}$ $\Gamma$-converges to
\[
\tilde{\mathcal E}_p:=\left\{\begin{array}{ll}
\mathcal E_p & \text{on $\mathcal C_p$}\\
\\
+\infty & \text{on $\mathcal C'_p \setminus \mathcal C_p$}
\end{array}\right.
\]
as $h\to+\infty$ with respect to the weak topology of $\mathcal W_p$, namely:
\begin{itemize}
    \item[(a)] {\bf(liminf inequality)} For any $(\vect{r},\vect{X}) \in \mathcal{C}'_p$ and for any $(\vect{r}_h,\vect{X}_h) \in \mathcal{C}'_p$ such that $(\vect{r}_h,\vect{X}_h) \stackrel{}{\rightharpoonup} (\vect{r},\vect{X})$ weakly in $\mathcal W_p$, it holds
\begin{equation}\label{liminf}
\tilde{\mathcal E}_p[\vect{r},\vect{X}]\leq \liminf_{h\to+\infty} \mathcal{E}_p^{\lambda_{h}} [\vect{r}_h,\vect{X}_h].
\end{equation}
 \item[(b)]{\bf(limsup inequality)} For any $(\vect{r},\vect{X}) \in \mathcal{C}'_p$ there exists $(\vect{\overline{r}}_h, \vect{\overline{X}}_h) \in \mathcal{C}'_p$ with $(\vect{\overline{r}}_h, \vect{\overline{X}}_h) \stackrel{}{\rightharpoonup} (\vect{r},\vect{X})$ weakly in $\mathcal W_p$ such that
\begin{equation}\label{limsup}
\tilde{\mathcal E}_p[\vect{r},\vect{X}]\geq\limsup_{h\to+\infty} \mathcal{E}_p^{\lambda_{h}}[\vect{\overline{r}}_h, \vect{\overline{X}}_h].
\end{equation}
\end{itemize}
\end{proposition}

\begin{proof}

Let $(\vect{r}_h, \vect{X}_h) \in \mathcal{C}'_p$ with $\mathcal{E}_p^{\lambda_h}[\vect{r}_h, \vect{X}_h]\leq c$, for some $c>0$. Since $\mathcal E_p \le \mathcal E_p^{\lambda_h}$ we can apply a similar argument as in Lemma \ref{compactness} and then, up to a subsequence, $(\vect{r}_h, \vect{X}_h) \rightharpoonup  (\vect{r}, \vect{X})$ weakly in $\mathcal W_p$, for some $(\vect r,\vect X) \in \mathcal W_p$ where $\vect r$ satisfies conditions \eqref{eq:r_0} - \eqref{eq:rprimo0} - \eqref{eq:velocita_1}. Next,
\[
\int_0^{2 \pi} \abs{\hat{\vect X}_h\circ c - \vect r_h}^2\, d\vartheta \le \frac{c}{\lambda_h}.
\]
Passing to the limit as $h\to+\infty$ and using \eqref{lsc-trace} we obtain
\[
\int_0^{2 \pi} \abs{\hat{\vect X}\circ c  - \vect r}^2\, d\vartheta=0
\]
which implies the trace constraint \eqref{eq:traccia}. In particular, $(\vect r,\vect X) \in \mathcal C_p$. Let us now show the liminf inequality \eqref{liminf}. Let $(\vect{r},\vect{X}) \in \mathcal{C}'_p$ and for any $(\vect{r}_h,\vect{X}_h) \in \mathcal{C}'_p$ such that $(\vect{r}_h,\vect{X}_h) \stackrel{}{\rightharpoonup} (\vect{r},\vect{X})$ weakly in $\mathcal W_p$. Up to extracting a subsequence, we can suppose that the liminf inequality \eqref{liminf} is actually a limit. As a consequence, we can assume that $\mathcal{E}_p^{\lambda_h}[\vect{r}_h, \vect{X}_h]\leq c$ for some $c>0$, otherwise \eqref{liminf} becomes trivially satisfied. Then $(\vect r_h,\vect X_h) \in \mathcal C_p$ and $(\vect r,\vect X) \in \mathcal C_p$ as well. Therefore \eqref{liminf} follows from the weak lower semicontinuity of $\mathcal E_p$, i.e.
\[
\mathcal E_p[\vect{r},\vect{X}]\leq \liminf_{h\to+\infty}\mathcal E_p[\vect{r}_h,\vect{X}_h]\leq \liminf_{h\to+\infty} \mathcal{E}_p^{\lambda_{h}}[\vect r_h,\vect X_h].
\]
In order to prove the limsup inequality \eqref{limsup}, for any $(\vect r,\vect X) \in \mathcal C_p'$ we simply choose $(\vect{\overline{r}}_h, \vect{\overline{X}}_h)=(\vect r,\vect X)$. If $(\vect r,\vect X) \in \mathcal C_p$ then $\mathcal E_p^{\lambda_h}[\vect{\overline{r}}_h, \vect{\overline{X}}_h]=\mathcal E_p[\vect r, \vect X]$ from which
\[
\mathcal E_p[\vect r, \vect X]=\limsup_{h\to+\infty}\mathcal E_p^{\lambda_h}[\vect{\overline{r}}_h, \vect{\overline{X}}_h].
\]
On the other hand, if $(\vect r,\vect X) \in \mathcal C'_p\setminus \mathcal C_p$, then
\[
\int_0^{2 \pi} \abs{\hat{\vect{\overline X}}_h\circ c  - \vect{\overline r}_h}^2\, d\vartheta=\int_0^{2 \pi} \abs{\hat{\vect X}\circ c - \vect r}^2\, d\vartheta=m>0.
\]
Then
\[
\limsup_{h\to+\infty}\mathcal E_p^{\lambda_h}[\vect{\overline{r}}_h, \vect{\overline{X}}_h]\ge m\limsup_{h\to+\infty}\lambda_h=+\infty,
\]
and this yields the conclusion.
\end{proof}

We are now in position to show the main result of this section. We first introduce
\[
\mathcal M_p=\{\text{$(\vect r,\vect X) \in \argmin\mathcal E_p$ : there exist $\lambda_h\to +\infty$ and $(\vect r_h,\vect X_h) \in \argmin \mathcal E_p^{\lambda_h}$ with $(\vect r_h,\vect X_h)\rightharpoonup (\vect r,\vect X)$ weakly in $\mathcal W_p$}\},
\]
and we show that

\begin{theorem}\label{thm:system_tot}
Let $(\vect r,\vect X) \in \mathcal M_p$. Then, $\vect r \in C^3([0,2\pi];\R^3)$, and 
\begin{equation}\label{eq:system_tot2}
2\vect r'''+2|\vect r''|^2\vect r'+\mu \vect F_\perp, \quad \text{on $[0,2\pi]$}
\end{equation}
where 
\begin{equation}\label{eq:F_def}
    \vect F(\theta)=\int_0^\theta(\textgoth{C}\nabla \vect{X} \circ c)\vect{\nu}_D\,d\theta,
\end{equation}
and the subscript $\perp$ stands for the vertical part with respect to $\vect r'$. 
\end{theorem}

\begin{proof}
Let $\lambda_h\to+\infty$ and let $(\vect r_h,\vect X_h) \in \argmin \mathcal E_p^{\lambda_h}$ be such that $(\vect r_h,\vect X_h) \rightharpoonup (\vect r,\vect X)$ weakly in $\mathcal W_p$.
We divide the proof into some steps. 
\\
\\
{\it Step 1}. We claim that 
\begin{equation}\label{step1}
\int_0^{2 \pi} 2\vect{r}''\cdot \vect{\eta}''-\mu \vect F \cdot \vect \eta'\,d\theta=0,
\end{equation}
for every $\vect \eta \in C_c^2([0,2\pi];\R^3)$ be such that $\vect r'\cdot \vect\eta' = 0$ on $[0, 2\pi]$. For any $h\in\N$ let $\vect \eta_h \in C_c^2([0,2\pi];\R^3)$ be such that
\begin{itemize}
	    \item[(i)] $\vect r_h'\cdot \vect\eta_h' = 0$ on $[0, 2\pi]$;
	    \item[(ii)] $\vect \eta_h\to \vect \eta$ strongly in $C^2$.
	\end{itemize}	
Let $\varepsilon>0$ and $\{\vect r_h^t : t\in (-\varepsilon,\varepsilon)\}$ be the family of curves as in Lemma \ref{lemma:variazione_curva} relatively to $\vect r_h$ and $\vect \eta_h$, namely:
\begin{itemize}
	    \item[(a)] $\vect{r}_h^0 = \vect r_h$ on $[0,2\pi]$;
	    \item[(b)] $\vect{r}_h^t$ satisfies conditions \eqref{eq:r_0} - \eqref{eq:rprimo0} - \eqref{eq:velocita_1} for any $t\in (-\varepsilon,\varepsilon)$;
	    \item[(c)] it holds
	    \[
	    \frac{\partial}{\partial t}|(\vect r_h^t)''|^2_{\big|_{t=0}}=2\vect r_h''\cdot \vect\eta_h'', \quad \text{a.e.\,on $(0,2\pi)$};
	    \]
	   \item[(d)] for any $\vect v,\vect w\in \R^3$ it holds
	    \[
	    \frac{\partial}{\partial t}|\vect v+t\vect w-\vect r_h^t|^2_{\big|_{t=0}}=2(\vect v-\vect r_h)\cdot (\vect w-\vect\eta_h), \quad \text{on $[0,2\pi]$}.
	    \]
	\end{itemize}
Let $\vect V,\vect V_h \in C^1(\overline D;\R^3)$ be such that
\[
\textrm{$\vect V_h \to \vect V$ strongly in $C^1$ and $\hat{\vect V}_h\circ c  = \vect \eta_h$ on $[0, 2\pi]$}.
\]
In particular, $\hat{\vect V}\circ c  = \vect \eta$ on $[0, 2\pi]$. Finally, for any $t\in (-\varepsilon,\varepsilon)$ let 
\[
    \label{eq:variation_X}
    \vect{X}_h^t= \vect{X}_h + t \vect V_h.
\]
To derive \eqref{eq:system_tot2}, we need to compute
\begin{equation*}
\begin{aligned}
\frac{{\rm d}}{{\rm d t}} \mathcal E_p^{\lambda_h}[\vect{r}_h^t,\vect{X}_h^t]_{\big|_{t = 0}}
&= \frac{{\rm d}}{{\rm d t}}\bigg\{\underbrace{\int_0^{2 \pi}|(\vect r_h^t)''|^2\,d\vartheta}_{\mathcal{I}_1} \\
&\quad + \frac{\mu}{2}\underbrace{\int_D \textgoth{C}\nabla \vect{X}_h^t:\nabla \vect{X}_h^t \,du dv}_{\mathcal{I}_2} + \lambda_h\underbrace{\int_0^{2 \pi}|\hat{\vect X_h^t}\circ c -\vect r_h^t|^2\, d\vartheta}_{\mathcal{I}_3}\bigg\}_{\big|_{t = 0}},
\end{aligned}
\end{equation*}
and we are going to consider each term separately. Let us start with $\mathcal{I}_1$. We immediately obtain, using (c),
\[
\frac{{\rm d}}{{\rm d t}}{\mathcal{I}_1}_{\big|_{t=0}}= 2 \int_0^{2 \pi} \vect{r}_h''\cdot \vect{\eta}_h''\, d\vartheta.
\]
Concerning the membrane contribution we get
\[
\frac{{\rm d}}{{\rm d t}}{\mathcal{I}_2}_{\big|_{t=0}} = 2\int_D \textgoth{C}\nabla \vect{X}_h: \nabla \vect{V}_h\, dudv.
\]
In order to compute the derivative of $\mathcal I_3$, we use (d) choosing as $\vect v=\hat{ \vect X}_h(c(\theta))$ and $\vect w=\vect\eta_h(\theta)$. Hence, we obtain 
\[
\begin{aligned}
\frac{{\rm d}}{{\rm d t}}{\mathcal{I}_3}_{\big|_{t=0}} &=\frac{{\rm d}}{{\rm d t}}\int_0^{2 \pi}|\hat{ \vect X}_h(c(\theta))+t\,\hat{ \vect V}_h(c(\theta))-\vect r_h(\theta)|^2\,d\theta_{\big|_{t=0}}\\
&=\frac{{\rm d}}{{\rm d t}}\int_0^{2 \pi}|\hat{ \vect X}_h(c(\theta))+t\,\vect{\eta}_h(\theta)-\vect r_h(\theta)|^2\,d\theta_{\big|_{t=0}}\\
&=2\int_0^{2 \pi}(\hat{ \vect X}_h(c(\theta))-\vect r_h(\theta))\cdot (\vect\eta_h(\theta) -\vect \eta_h(\theta))\,d\theta=0.
\end{aligned}
\]
Collecting all the terms we get 
\[
\begin{aligned}
    \frac{{\rm d}}{{\rm d t}} \mathcal{E}_p^{\lambda_h}[\vect{r}_h^t,\vect{X}_h^t]_{\big|_{t = 0}} = 2 \int_0^{2 \pi} \vect{r}_h''\cdot \vect{\eta}_h''\, d\vartheta +\mu \int_D\textgoth{C} \nabla \vect{X}_h: \nabla \vect{V}_h\, dudv.
\end{aligned}
\]
Since $(\vect r_h,\vect X_h) \in \argmin \mathcal{E}_p^{\lambda_h}$ we have 
\[
\frac{{\rm d}}{{\rm d t}} \mathcal{E}_p^{\lambda_h}[\vect{r}_h^t,\vect{X}_h^t]_{\big|_{t = 0}}=0,
\]
leading to the condition that $\forall\,h \in \N$,
\[
2 \int_0^{2 \pi} \vect{r}_h''\cdot \vect{\eta}_h''\, d\vartheta +\mu \int_D\textgoth{C} \nabla \vect{X}_h: \nabla \vect{V}_h\, dudv=0.
\]
Passing to the limit as $h\to+\infty$ we obtain
\begin{equation}\label{ELhlim}
2 \int_0^{2 \pi} \vect{r}''\cdot \vect{\eta}''\, d\vartheta +\mu \int_D\textgoth{C} \nabla \vect{X}: \nabla \vect{V}\, dudv=0.
\end{equation}
Integrating by parts we reduce \eqref{ELhlim} to 
\[
\begin{aligned}
    0 &=2 \int_0^{2 \pi} \vect{r}''\cdot \vect{\eta}''\, d\vartheta +\mu \int_D \textgoth{C}\nabla \vect{X}: \nabla \vect{V}\, dudv\\
    &=2 \int_0^{2 \pi} \vect{r}''\cdot \vect{\eta}''\, d\vartheta -\mu\int_D {\rm div}\,(\textgoth{C}\nabla \vect{X})\cdot \vect V\, dudv + \mu\int_{\partial D} \textgoth{C}\nabla \vect{X}(\vect p)\,\vect p\cdot \hat {\vect V}(\vect p)\,d\mathcal H^1(\vect p)\\
&=\int_0^{2 \pi} 2\vect{r}''\cdot \vect{\eta}''+\mu(\textgoth{C}\nabla \vect{X} \circ c)\vect{\nu}_D\cdot \vect \eta\,d\theta-\mu\int_D {\rm div}\,(\textgoth{C}\nabla \vect{X}) \cdot \vect V\, dudv\\
&\stackrel{\eqref{eq:solo_superficie}}{=}\int_0^{2 \pi} 2\vect{r}''\cdot \vect{\eta}''+\mu(\textgoth{C}\nabla \vect{X} \circ c)\vect{\nu}_D\cdot \vect \eta\,d\theta\\
&\stackrel{\eqref{eq:F_def}}{=}\int_0^{2\pi}2\vect r''\cdot \vect \eta''-\mu \vect F\cdot \vect \eta'\,d\theta
\end{aligned}
\]
which completes the proof of \eqref{step1}.
\\
\\
{\it Step 2}. We claim that $\vect r\in W^{3,2}((0,2\pi);\R^3)$ and 
\begin{equation}\label{step3}
(2\vect{r}'''+\mu \vect F)_\perp=0.
\end{equation}
In the following, we employ a similar argument to the one introduced in Section 3.2 of \cite{bernatzki2001minimal}. Since $\vect r$ is a closed curve of class $C^1$ there exist $\vect v_i \in C_c^1([0,2\pi];\R^3)$, $i=1,2,3$, with $\vect v_i \cdot \vect r'=0$, such that the vectors $\vect Z_i$ given by
\[
\vect Z_i=\int_0^{2\pi}\vect v_i\,d\theta
\]
generate an orthonormal basis in $\R^3$. Let $\vect \xi \in C_c^\infty((0,2\pi);\R^3)$ and
\[
a_i=\int_0^{2\pi}\vect Z_i \cdot\vect \xi_\perp\,d\theta, \quad i=1,2,3.
\]
Let 
\[
\vect \psi=\vect \xi_\perp-\sum_{i=1}^3a_i\vect v_i, \quad \vect \phi(\theta)=\int_0^\theta \vect \psi\,ds.
\]
By construction, $\vect \phi$ is an admissible test, since $\vect \phi \in C_c^2((0,2\pi);\R^3)$ and $\vect r'\cdot \vect \phi'=0$. As a consequence, it can be substituted in \eqref{step1}, i.e. 
\[
\int_0^{2 \pi} 2\vect{r}''\cdot \vect{\phi}''-\mu \vect F \cdot \vect \phi'\,d\theta=0.
\]
Using $\vect \xi_\perp=\vect\xi-(\vect \xi \cdot \vect r')\vect r'$ and recalling that $|\vect r'|=1$, the previous equation reads as  
\[
\int_0^{2\pi}2\vect r''\cdot \vect \xi'\,d\theta=\int_0^{2\pi}2|\vect r''|^2\vect \xi \cdot\vect r'+2\sum_{i=1}^3a_i\vect r''\cdot \vect v_i'+\mu \vect F \cdot \vect \xi_\perp-\mu \sum_{i=1}^3a_i\vect F \cdot \vect v_i\,d\theta.
\]
The right-hand side can be easily bounded in $L^2$ since all the terms have higher regularity, namely there exists a constant $c>0$ such that 
\[
\left|\int_0^{2\pi}2\vect r''\cdot \vect \xi'\,d\theta\right|\le c\|\vect \xi\|_{L^2}, \quad \forall \vect \xi \in C_c^\infty((0,2\pi);\R^3),
\]
obtaining that $\vect r''\in W^{1,2}((0,2\pi);\R^3)$, hence $\vect r \in W^{3,2}((0,2\pi);\R^3)$. Integrating by parts \eqref{step1}, we thus get
\[
    \int_0^{2 \pi} (2\vect{r}'''+\mu \vect F)\cdot \vect \phi'\,d\theta=0,
\]
which means that 
\begin{equation}\label{finalstep3}
\int_0^{2 \pi} (2\vect{r}'''+\mu \vect F)_\perp \cdot \vect \xi\,d\theta=\int_0^{2 \pi} (2\vect{r}'''+\mu \vect F)_\perp\cdot \sum_{i=1}^3a_i\vect v_i\,d\theta.
\end{equation}
We choose three Lebesgue points $\theta_1,\theta_2,\theta_3 \in [0,2\pi]$ for $(2\vect{r}'''+\mu \vect F)_\perp$ and sequences $(\vect{v}_i^h)$ such that $\vect{v}_i^h\in C_c^1([0,2\pi];\R^3)$, $i=1,2,3$ and $h\in \N$, $ \vect{v}_i^h \cdot \vect r'=0$,
\[
\vect Z_i^h=\int_0^{2\pi}\vect v_i^h\,d\theta
\]
form an orthonormal basis in $\R^3$, and 
\[
a_i^h\vect v_i^h \to w_i\delta_{\theta_i}, \quad \text{in $\mathcal D'$},
\]
where 
\[
a_i^h=\int_0^{2\pi}\vect Z_i^h \cdot\vect \xi_\perp\,d\theta, \quad i=1,2,3,\,h\in \N,
\]
and $w_i \cdot (2\vect{r}'''(\theta_i)+\mu \vect F(\theta_i))_\perp=0$. Putting $\vect v_i^h$ in \eqref{finalstep3} and passing to the limit as $h\to+\infty$ we obtain 
\[
\int_0^{2 \pi} (2\vect{r}'''+\mu \vect F)_\perp \cdot \vect \xi\,d\theta=0,
\]
which gives the thesis \eqref{step3}, by the arbitrariness of $\vect \xi$.
\\
\\
{\it Step 3.} We can now conclude the proof. First of all, notice that since $|\vect r'|=1$ we get
\[
\vect r''' \cdot \vect r'=(\vect r'' \cdot \vect r')'-|\vect r''|^2=-|\vect r''|^2.
\]
This means that $(\vect r''')_\perp=\vect r'''+|\vect r''|^2\vect r'$, and then \eqref{eq:system_tot2} follows from \eqref{step3}. We also observe that \eqref{eq:system_tot2} implies that $\vect r^{iv}\in L^2((0,2\pi);\R^3)$, that is $\vect r\in W^{4,2}((0,2\pi);\R^3)\subset C^3([0,2\pi];\R^3)$ and this concludes the proof.  
\end{proof}

\begin{remark}\label{rem_elastica}
We point out that using the parametric approach we are not able to consider non-linear energies of the form 
\[
\int_0^{2 \pi} f\left(\vect{r}''\right)\, d\vartheta + \int_D \Psi(\nabla \vect X)\,dudv.
\]
More precisely, imposing standard conditions on $f$ and $\Psi$ it is possible to prove existence of minimizers and it is also possible to get a $\Gamma$-convergence result similar to Proposition \ref{th:gamma_conv}. Nevertheless, in the computation of the Euler-Lagrange equations we are not able, in general, to pass to the limit, as $h\to+\infty$, in the expression
\[
 \int_0^{2 \pi} f'\left(\vect{r}''_h\right) \cdot \vect{\eta}''\, d\theta+ \int_D \Psi(\nabla \vect X_h) : \nabla \vect V\,dudv
\]
which comes from the variation of $(\vect{r}_h, \vect{X}_h)$, since we have just weak convergence of $\vect{r}''_h$ and $\nabla \vect{X}_h$. In order to treat more general energies, both for the elastic curve and for the membrane, we need to pass to the framed curves framework, which has already been exploited in \cite{bevilacqua2021variational} only for the elastic curves.
\end{remark}

\section{Framed curve approach}\label{sec:second-approach}
Using the {\it framed curve approach} (introduced by Gonzalez et al.\,\cite{GMSM} and developed in \cite{bevilacqua2021variational} for elastic curves), in this section we are able to consider a more general situation. Let $p>1$ and let $(\vect{t}|\vect{n}|\vect{b})\in W^{1,p}((0,2 \pi);SO(3))$, where $SO(3)$ is the set of all $3\times3$ rotation matrices (namely, $\{\vect t,\vect n,\vect b\}$ is an orthonormal positively oriented basis in $\R^3$). On such a triple, we impose the following constraints:
\begin{align}
    \label{eq:con_1}\vect{t}'\cdot \vect{b} &= 0, \quad \text{a.e.\,on $(0,2\pi)$},\\
    \label{eq:con_2}\int_0^{2 \pi} \vect{t}\, d\vartheta &= 0,\\
    \label{eq:con_3}\vect{t}(2 \pi)&=\vect{t}(0).
\end{align}
We define the set of constraints
\[
    \mathcal{C}_{f}^{\rm r} =\left\{(\vect{t}|\vect{n}|\vect{b}) \in W^{1,p}((0,2 \pi);SO(3)): \, \eqref{eq:con_1} - \eqref{eq:con_2} - \eqref{eq:con_3} \hbox{ hold true}\right\},
\]
where the subscript $f$ refers to the {\em framed curve approach} adopted here. The elastic energy of the frame is defined as follows: let $f\colon \R\times \R \to \R$ be a measurable function and consider $\mathcal E_f^{\rm r} \colon \mathcal C_f^{\rm r}\to [-\infty,+\infty]$ as
\[
    \mathcal{E}_{f}^{\rm r}\left[(\vect{t}|\vect{n}|\vect{b}) \right]= \int_0^{2 \pi} f(\kappa,\tau)\, d\vartheta,
\]
where $\kappa=\vect{t}'\cdot \vect{n}$ and $\tau=\vect{n}'\cdot \vect{b}$, for a.e.\,on $(0,2\pi)$.

\begin{remark}
Fix $(\vect t|\vect n|\vect b)\in W^{1,p}((0,2 \pi);SO(3))$ and $\vect x_0 \in \R^3$. We reconstruct the curve $\vect  r_{\vect x_0}\colon [0,2\pi]\to \R^3$ clamped at ${\vect x}_0$ and {\it generated} by the orthonormal frame $\{{\vect t},{\vect n},{\vect b}\}$ by means of  
\[
{\vect r}_{{\vect x}_0}(\theta)={\vect x}_0+\int_0^\theta{\vect t}\,ds.
\]
As observed in detail in \cite{bevilacqua2021variational}, ${\vect r}_{{\vect x}_0} \in W^{2,p}((0,2\pi);\R^3) $ is parametrized by the arclength, it is a closed curve and the tangent vector to ${\vect r}_{{\vect x}_0}$ is continuous. Moreover, the quantities $\kappa,\tau$ represent the {\it (signed) weak curvature} and the {\it weak torsion} of ${\vect r}_{{\vect x}_0}$ respectively. Notice that from \eqref{eq:con_1} we deduce that the following Serret-Frenet type vector system of ODEs
\begin{equation}\label{frenet}
\left\{
\begin{aligned}
&\vect{t}' = \kappa \vect{n},\\
&\vect{n}' = -\kappa \vect{t} + \tau \vect{b},\\
&\vect{b}' = -\tau \vect{n},
\end{aligned}
\right.
\end{equation}
holds true.
\end{remark}

The energy contribution of the membrane is defined as follows: take $\vect{X} \in W^{1,2}(D;\R^3)$ satisfying the trace constraint, i.e.
\begin{equation}
\label{eq:trace_constraint}
\hat{\vect X}\circ c = \vect{r}_{\vect{x}_0}, \quad \text{a.e.\,on $(0,2\pi)$}.
\end{equation}
We let
\[
    \mathcal{C}_{f}^{\rm m} = \left\{\vect{X} \in W^{1,2}(D;\R^3): \eqref{eq:trace_constraint}\, \hbox{ holds true}\right\},
\]
and the energy functional $\mathcal{E}_{f}^{\rm m}\colon  \mathcal{C}_f^{\rm m} \to [-\infty,+\infty]$ is given by
\[
    \mathcal{E}_{f}^{\rm m}\left[\vect{X}\right] = \frac{\mu}{2}\int_D \Psi(\nabla \vect{X})\, dudv,
\]
where $\mu >0$ is the shear modulus and $\Psi \colon \R^{3\times 2}\to \R$ is continuous. In order to define the total energy, we let 
\[
\mathcal{C}_f := \left\{\left((\vect{t}|\vect{n}|\vect{b}), \vect{X}\right) \in W^{1,p}((0,2 \pi);\R^3) \times W^{1,2}(D;\R^3): \, \eqref{eq:con_1} - \eqref{eq:con_2} - \eqref{eq:con_3} - \eqref{eq:trace_constraint} \hbox{ hold true}\right\},
\]
be the set of constraints. Hence the total energy $\mathcal E_f \colon \mathcal C_f\to [-\infty,+\infty]$ is given by 
\[
\mathcal E_f[(\vect t|\vect n|\vect b),\vect X]=\mathcal E_f^{\rm r}[(\vect t|\vect n|\vect b)]+\mathcal E_f^{\rm m}[\vect X] = \int_0^{2 \pi} f(\kappa,\tau)\, d\vartheta + \frac{\mu}{2} \int_D \Psi(\nabla \vect{X})\, dudv.
\]

\subsection{Existence of minimizers}
\label{subsec:minimizers_framed}
In this section, we prove the existence of minimizers of $\mathcal E_f$. First of all, we need to introduce a definition \cite{morrey1952quasi}.

\begin{definition}
\label{def:quasiconvex}
Let $\Omega$ be an open set and $\Phi\colon \R^{k\times n}\to \R$ be a continuous function. We say that $\Phi$ is a {\em quasi-convex} function if $\forall\, \tens{A} \in \R^{k\times n}$ and $\forall\, \vect{\xi} \in C^{1}_c(\Omega; \R^k)$ the following inequality holds
\[
    \int_\Omega \Phi\left(\tens{A} + \nabla \vect{\xi}\right) \, dx\geq\abs{\Omega}\Phi\left(\tens{A}\right),
\]
where $|\Omega|$ is the Lebesgue measure of $\Omega$.
\end{definition}
\begin{theorem}
\label{th:esistenza_frame}
Assume that $f\colon \R\times \R \to \R$ satisfies the following hypotheses:
\begin{itemize}
    \item[(i)] $f$ is continuous and convex;
    \item[(ii)] there exist $c_1,c_2 >0$ and $c_3 \in \R$ such that $f(a,b)\geq c_1 \abs{a}^p + c_2 \abs{b}^p + c_3$ for all $a,b \in \R$ and for some $p>1$.
\end{itemize}
 Moreover, we require for the elastic energy density $\Psi \colon \R^{3\times 2} \to \R$, that:
\begin{itemize}
    \item[(i)] $\Psi$ is quasi-convex;
    \item[(ii)] there exist $c_4,c_5>0$ such that $c_4|\tens A|^q\leq \Psi(\tens{A}) \leq c_5 \abs{\tens{A}}^q$ for all $\tens A \in \R^{3\times 2}$ and for some $q>1$.
\end{itemize}
Then $\mathcal{E}_f$ has a minimizer.
\end{theorem}
\begin{proof}
The proof can be easily obtained combining \cite[Thm.\,3.1]{bevilacqua2021variational}, Lemma \ref{compactness} and the lower semicontinuity result \cite[Thm.\,II.4]{acerbi1984semicontinuity}.
\end{proof}

\subsection{First-order necessary conditions}
In this section, using the argument as in \cite{bevilacqua2021variational}, we want to derive the first-order necessary conditions for minimizers of $\mathcal{E}_f$ applying the infinite-dimensional version of the Lagrange multipliers’ method (see, for instance, \cite{zeidler2012applied}, Sec. 4.14). Precisely, we are going to use the following abstract result.

\begin{theorem}\label{th:inf_dim_lag_multiplier}
Let $Y,Z$ be two real Banach spaces, $\mathcal F\in C^1(Y)$ and $\mathcal G \in C^1(Y;Z)$. Let $y_0 \in Y$ be such that
	\[
	\mathcal F(y_0)=\min\{\mathcal F(y) \colon y\in Y\} \quad \hbox{ and } \quad \mathcal{G}(y_0) = 0.
	\]
	Assume that $\mathcal G'(y_0) \colon Y\to Z$ is surjective.
	Then there exists a Lagrange multiplier $\lambda \colon Z \to \R$, a linear and continuous application, such that
	\[
	\mathcal F'(y_0)=\lambda\mathcal G'(y_0).
	\]
\end{theorem}

Applying Theorem \ref{th:inf_dim_lag_multiplier} we can prove the following result.

\begin{theorem}
\label{th:lagrange_multiplier_nostro}
Assume that $f$ is of class of $C^1$ and that there is $\alpha>0$ such that 
\begin{equation*}
    \begin{aligned}
    &\abs{f_a(a,b)} \leq \alpha\left(1 + \abs{a}^{p-1} +\abs{b}^{p-1}\right),&&&\abs{f_b(a,b)} \leq \alpha\left(1 + \abs{a}^{p-1} +\abs{b}^{p-1}\right),
    \end{aligned}
\end{equation*}
for all $a,b \in \R$. In addition, we assume that there is $\beta>0$ such that 
\[
\abs{{\rm D}\Psi(\tens{A})} \leq \beta \left(1 + \abs{\tens{A}}^2\right)
\]
for all $\tens{A} \in \R^{3\times 2}$. Let $\left((\vect{t}|\vect{n}|\vect{b}), \vect{X}\right)$ be a minimizer for $\mathcal{E}_f$. Then, there exist $\omega \in L^p(0, 2 \pi)$, $\vect{\chi} \in L^2((0,2 \pi);\R^3)$ and $\vect{\lambda} \in \R^3$ such that the following first-order necessary conditions hold a.e. $(u,v) \in D$ and $\vartheta \in (0, 2 \pi)$:
\begin{empheq}[left=\empheqlbrace]{align}
\label{eq:armonica_th}
&\diver\left({\rm D}\Psi(\nabla \vect{X})\right) = 0, &&\hbox{on }D,\\
 \label{eq:uno_th}
 &-f_a' -\omega \tau -\vect n \cdot \int_0^\vartheta \vect{\chi} \, ds = \vect{\lambda}\cdot \vect{n}, &&\hbox{on $(0,2\pi)$},\\
 \label{eq:due_th}
 & -f_a \tau +\omega' + f_b \kappa -\vect b \cdot \int_0^\vartheta \vect{\chi}\, ds= \vect{\lambda}\cdot \vect{b}, &&\hbox{on }(0,2\pi),\\
\label{eq:tre_th}
	&\frac{\mu}{2}\left({\rm D}\Psi(\nabla \vect{X}) \circ c\right)\vect{\nu}_D = \vect{\chi}, &&\hbox{on }(0,2\pi),\\
 \label{eq:quattro_th}
 & \omega \kappa -f'_b = 0,&&\hbox{on }(0,2\pi),
\end{empheq}
where here, for simplicity, $f_a:=f_a(\kappa,\tau)$ and $f_b:=f_b(\kappa,\tau)$.
\end{theorem}
\begin{proof}
In order to apply Theorem \ref{th:inf_dim_lag_multiplier} we introduce the Banach spaces
\[
\begin{aligned}
&Y := \left\{\left((\vect{t}|\vect{n}|\vect{b}), \vect{X}\right) \in W^{1,p}((0,2 \pi); \R^3) \times W^{1,2}(D; \R^3):\, \int_0^{2 \pi} \vect{t}\, d\theta =0, \,\vect{t}(0) = \vect{t}(2 \pi) \right\},\\
&Z :=L^p(0,2 \pi))\times L^p(0,2 \pi)\times L^p(0,2 \pi) \times L^p((0,2 \pi);\R^3) \times L^p(0,2 \pi) \times L^2((0,2 \pi);\R^3)
\end{aligned}
\]
and the functionals $\mathcal F \colon Y \to \R$ and $\mathcal G \colon Y \to Z$ given by 
\[
\begin{aligned}
\mathcal{F}\left[(\vect{t}|\vect{n}|\vect{b}), \vect{X}\right] &=\int_0^{2 \pi} f(\vect{t}'\cdot \vect{n}, \vect{n}'\cdot \vect{b})\, d\vartheta + \frac{\mu}{2} \int_D \Psi(\nabla \vect{X})\,dudv,\\
\mathcal{G}\left[(\vect{t}|\vect{n}|\vect{b}), \vect{X}\right] &= \left(\vect{t}\cdot \vect{t}-1, \vect{n}\cdot \vect{n}-1, \vect{t}\cdot \vect{n}, \vect{b}- \vect{t}\times \vect{n}, \vect{t}'\cdot \vect{b}, \hat{\vect X}- \vect{r}_{\vect{x}_0}\right).
\end{aligned}
\]
To derive the first-order necessary conditions for minimizers, we fix $\left(\left(\vect{\eta}_1| \vect{\eta}_2|\vect{\eta}_3\right), \vect{V}\right) \in Y$. Hence we have
\[
\begin{aligned}
&\frac{{\rm d}}{{\rm d}\sigma} \mathcal{F}\left[(\vect{t} + \sigma \vect{\eta}_1|\vect{n} + \sigma \vect{\eta}_2|\vect{b} + \sigma \vect{\eta}_3),\vect{X}+ \sigma \vect{V}\right]_{\big|_{\sigma = 0}}\\
&\stackrel{\cite{bevilacqua2021variational}}{=} \int_0^{2 \pi} f_a\,\left(\vect{\eta}_1'\cdot \vect{n} + \vect{t}'\cdot \vect{\eta}_2\right)\, d\vartheta+ \int_0^{2 \pi} f_b\,\left(\vect{\eta}_2'\cdot \vect{b} + \vect{n}'\cdot \vect{\eta}_3\right)\, d\vartheta +\frac{\mu}{2} \int_D {\rm D}\Psi(\nabla\vect{X}):\nabla \vect{V}\, dudv.
\end{aligned}
\]
Let us define
\[
    L\left[(\vect{\eta}_1, \vect{\eta}_2,\vect{\eta}_3),\vect{V}\right]:= \int_0^{2\pi} f_a\,\left(\vect{\eta}_1'\cdot \vect{n} + \vect{t}'\cdot \vect{\eta}_2\right)\, d\vartheta+ \int_0^{2 \pi} f_b\,\left(\vect{\eta}_2'\cdot \vect{b} + \vect{n}'\cdot \vect{\eta}_3\right)\, d\vartheta + \frac{\mu}{2} \int_D {\rm D}\Psi(\nabla\vect{X}):\nabla \vect{V}\, dudv.
\]
It is immediate to prove that $L$ is linear and continuous  obtaining that $\mathcal{F} \in C^1(Y)$ (see \cite{badiale2010semilinear} for details), and 
\[
\begin{aligned}
\mathcal{F}'\left[(\vect{t}, \vect{n}, \vect{b}),\vect{X}\right]\left[(\vect{\eta}_1, \vect{\eta}_2,\vect{\eta}_3),\vect{V}\right]&= \int_0^{2 \pi} f_a\,\left(\vect{\eta}_1'\cdot \vect{n} + \vect{t}'\cdot \vect{\eta}_2\right)\, d\vartheta+ \int_0^{2 \pi} f_b\,\left(\vect{\eta}_2'\cdot \vect{b} + \vect{n}'\cdot \vect{\eta}_3\right)\, d\vartheta\\
&\quad +\frac{\mu}{2} \int_D {\rm D}\Psi(\nabla\vect{X}):\nabla \vect{V}\, dudv\\
&=\int_0^{2 \pi} f_a\,\left(\vect{\eta}_1'\cdot \vect{n} + \vect{t}'\cdot \vect{\eta}_2\right)\, d\vartheta+ \int_0^{2 \pi} f_b\,\left(\vect{\eta}_2'\cdot \vect{b} + \vect{n}'\cdot \vect{\eta}_3\right)\, d\vartheta\\
&\quad + \frac{\mu}{2}\int_0^{2 \pi} \left({\rm D}\Psi(\nabla\vect{X})\circ c\right)\vect{\nu}_D\cdot (\hat{\vect V} \circ c)\, d\vartheta -\frac{\mu}{2} \int_D \diver\left({\rm D}\Psi(\nabla\vect{X})\right) \cdot \vect{V}\, dudv.
\end{aligned}
\]
Considering the constraints, first we immediately deduce that $\mathcal{G}\in C^{1}(Y;Z)$ since both the trace and the integral operators are linear and continuous. Hence we can define
$$
\begin{aligned}
\mathcal{G}'\left[(\vect{t}, \vect{n}, \vect{b}),\vect{X}\right] = \left(2\vect{t}\cdot \vect{\eta}_1, 2\vect{n}\cdot \vect{\eta}_2,\right.&\left. \vect{t}\cdot \vect{\eta}_2 +\vect{n}\cdot \vect{\eta}_1, \vect{\eta}_3 + \vect{n}\times \vect{\eta}_1 - \vect{t}\times \vect{\eta}_2,\right.\\
&\left.\vect{b}\cdot \vect{\eta}_1' + \vect{t}' \cdot \vect{\eta}_3, \hat{\vect V}\circ c- \int_0^\vartheta \vect{\eta}_1\, ds \right).
\end{aligned}
$$
It is straightforward to prove that $\mathcal{G}'$ is surjective. Take $a_1,a_2,a_3,a_5 \in L^p(0,2 \pi)$, $\vect{a}_4 \in L^p((0,2 \pi);\R^3)$ and $\vect{a}_6 \in L^2((0, 2 \pi); \R^3)$.
We can prove that $\zeta_j=\zeta_j(\vartheta)$, for $j=1,2$, can be determined in a way that
\[
\begin{aligned}
&\vect{\eta}_1 = \frac{a_1}{2}\vect{t}+ \zeta_1\vect{b}\\
&\vect{\eta}_2 = \zeta_2 \vect{t} + \frac{a_2}{2}\vect{b}\\
&\vect{\eta}_3 = \vect{a}_4 + \frac{a_1}{2} \vect{b} -\lambda_1\vect{t}- \frac{a_2}{2}\vect{n}\\
&\hat{\vect V} \circ c = \vect{a}_6 +\int_0^{\vartheta} \left(\frac{a_1}{2}\vect{t} + \zeta_1 \vect{b}\right)\, ds,
\end{aligned}
\]
satisfy $\mathcal G'\left[\left(\vect{t}|\vect{n}|\vect{b}\right), \vect{X}\right]\left(\left(\vect{\eta}_1|\vect{\eta}_2|\vect{\eta}_3\right),\vect{V}\right)=(a_1,a_2,a_3,\vect{a}_4,a_5,\vect{a}_6)$. Precisely, to get the value of $\vect{V}$ on the entire disc, we can just consider, for instance, its harmonic extension.
Hence, applying Theorem \ref{th:inf_dim_lag_multiplier}, we can say that there exist $\lambda_1, \lambda_2, \lambda_3 \in L^{p'}(0,2 \pi)$, $\vect{\lambda}_4 \in L^{p'}((0,2 \pi);\R^3)$, $\lambda_5 \in L^{p'}(0,2 \pi)$ and $\vect{\lambda}_6 \in L^{2}((0,2 \pi);\R^3)$ such that
\[
    \begin{aligned}
&\int_0^{2 \pi} f_a\vect{n}\cdot \vect{\eta}_1'\, d\vartheta + \int_0^{2 \pi} f_a \vect{t}'\cdot \vect{\eta}_2\, d\vartheta+ \int_0^{2 \pi} f_b  \vect{b}\cdot \vect{\eta}_2'\, d\vartheta +\int_0^{2 \pi} f_b\vect{n}'\cdot \vect{\eta}_3\, d\vartheta\\
&\qquad \qquad \qquad\quad+ \frac{\mu}{2}\int_0^{2 \pi} \left({\rm D}\Psi(\nabla\vect{X} \circ c)\right)\vect{\nu}_D\cdot \left(\hat{\vect V}\circ c\right)\, d\vartheta -\frac{\mu}{2} \int_D \diver\left({\rm D}\Psi(\nabla\vect{X})\right) \cdot \vect{V}\, dudv\\
=&\int_0^{2 \pi} \left(2\lambda_1 \vect{t} + \lambda_3 \vect{n} +\vect{\lambda}_4 \times \vect{n}\right)\cdot \vect{\eta}_1\, d\vartheta + \int_0^{2 \pi} \left(2 \lambda_2 \vect{n} + \lambda_3 \vect{t} -\vect{\lambda}_4 \times \vect{t}\right)\cdot \vect{\eta}_2\, d\vartheta+ \int_0^{2 \pi} \left(\vect{\lambda}_4 + \lambda_5 \vect{t}'\right)\cdot \vect{\eta}_3\, d\vartheta\\
&\qquad \qquad \qquad\qquad \qquad\qquad\qquad\;\,+ \int_0^{2 \pi} \lambda_5 \vect{b}\cdot \vect{\eta}_1'\, d\vartheta + \int_0^{2 \pi} \vect{\lambda}_6 \cdot \left[ \hat{\vect V}\circ c - \int_0^\vartheta \vect{\eta}_1\, ds \right]\, d\vartheta.
\end{aligned}
\]
Choosing in a suitable way the arbitrary tests $\varphi \in C_c^\infty(0,2\pi)$, we will deduce the first order necessary conditions. Indeed, using the fact that $\kappa =\vect{t}'\cdot \vect{n}$ and $\tau = \vect{n}'\cdot \vect{b}$, we obtain
\[
\left\{
\begin{aligned}
&\diver\left({\rm D}\Psi(\nabla\vect{X})\right) = 0 &&& \text{choosing $\vect{\eta}_1= \vect{\eta}_2 = \vect{\eta}_3 = \hat{\vect V} = 0$ and taking $\vect{V}$ arbitrary on $D$}\\
&f_b \vect{n}' = \vect{\lambda}_4 + \lambda_5 \vect{t}', &&&\text{choosing $\vect{\eta}_1= \vect{\eta}_2 = \vect{V} = 0$ and taking $\vect{\eta}_3$ arbitrary}\\
&\lambda_3 = 0, &&& \text{choosing $\vect{\eta}_1= \vect{\eta}_3 = \vect{V} = 0$ and taking $\vect{\eta}_2 = \varphi \vect t$}\\
&f_a \kappa + f_b \tau = 2 \lambda_2 -\vect{\lambda}_4 \cdot \vect{b}, &&& \text{choosing $\vect{\eta}_1= \vect{\eta}_3 = \vect{V} = 0$ and taking $\vect{\eta}_2 = \varphi \vect{n}$}\\
&- f_b' = \vect{\lambda}_4 \cdot \vect{n}, &&& \text{choosing $\vect{\eta}_1= \vect{\eta}_3 = \vect{V} = 0$ and taking $\vect{\eta}_2 = \varphi \vect{b}$.}
\end{aligned}
\right.
\]
Then we can derive the complete expressions of the Lagrange multipliers:
$$
\left\{
\begin{aligned}
&\lambda_3 = 0,\\
& \vect{\lambda}_4 = -f_b \kappa \vect{t} -f_b'\vect{n} + f_b \tau \vect{b},\\
&\lambda_2 = \frac{f_a \kappa + 2 f_b \tau}{2},\\
&\lambda_5 \kappa -f_b' = 0.
\end{aligned}
\right.
$$
Taking $\vect{\eta}_1 = \vect{\eta}_2 = \vect{\eta}_3 = 0$ and an arbitrary $\vect{V}$, we get
\begin{equation}
    \label{eq:lambda_6}
    \frac{\mu}{2} \left({\rm D}\Psi(\nabla \vect{X}\circ c)\right)\vect{\nu}_D = \vect{\lambda}_6.
\end{equation}
Finally, for $\vect{\eta}_2 = \vect{\eta}_3 = 0$ and $\vect{\eta}_1 = \hat{\vect V}\circ c$ and using \eqref{eq:lambda_6}, we obtain
\begin{align}
\nonumber
&\int_0^{2 \pi} \left(f_a \vect{n} -\lambda_5 \vect{b}\right)\cdot \vect{\eta}_1'\, d\vartheta = \int_0^{2 \pi} \left(2 \lambda_1 \vect{t} + \vect{\lambda}_4 \times \vect{n}\right)\cdot \vect{\eta}_1\, d\vartheta - \int_0^{2 \pi} \vect{\lambda}_6\cdot\left(\int_0^\vartheta \vect{\eta}_1\, ds \right)\, d\vartheta,\\\label{eq:forall_eta_1}
&\int_0^{2 \pi}\left[-\left(f_a \vect{n} -\lambda_5 \vect{b}\right)' -\left(2 \lambda_1 \vect{t} + \vect{\lambda}_4 \times \vect{n} \right)-\int_0^\vartheta \vect{\lambda}_6\, ds \right]\cdot \vect{\eta}_1\, d\vartheta = 0,
\end{align}
having used Fubini's theorem and the definition of weak derivative. Since \eqref{eq:forall_eta_1} has to be true for all $\vect{\eta}_1$ such that $\int_0^{2 \pi}\vect{t}\, d\vartheta= 0$, there exists a constant vector $\vect{\lambda}\in \R^3$ such that
$$
\begin{aligned}
-\left(f_a \vect{n} -\lambda_5 \vect{b}\right)' -2 \lambda_1 \vect{t} - \vect{\lambda}_4 \times \vect{n} - \int_0^\vartheta \vect{\lambda}_6\, ds = \vect{\lambda}.
\end{aligned}
$$
Using $\vect{\lambda}_4 =f_b \vect{n}' -\lambda_5 \vect{t}'$, $\vect{n}' = -\kappa \vect{t} + \tau \vect{b}$ and $\vect{b}' = -\tau \vect{n}$, we obtain
\[
    \begin{aligned}
     \vect{t}\bigg(f_a \kappa - 2 \lambda_1 + f_b \tau &-\vect t \cdot \int_0^\vartheta \vect{\lambda}_6\, ds\bigg) + \vect{n}\left(-f_a' -\lambda_5 \tau -\vect n \cdot \int_0^\vartheta \vect{\lambda}_6 \, ds\right) \\
     &+ \vect{b} \left(-f_a \tau +\lambda_5' + f_b \kappa - \vect b \cdot \int_0^\vartheta \vect{\lambda}_6\, ds\right) = \left(\vect{\lambda}\cdot \vect{t}\right)\vect{t}+\left(\vect{\lambda}\cdot \vect{n}\right)\vect{n}+\left(\vect{\lambda}\cdot \vect{b}\right)\vect{b}.
    \end{aligned}
\]
Renaming $\lambda_5 = \omega$ and $\vect{\lambda}_6 = \vect{\chi}$ we get the thesis.
\end{proof}
\begin{remark}
Theorem \ref{th:lagrange_multiplier_nostro} is similar to Theorem 4.2 of \cite{bevilacqua2021variational} where the statement of Theorem 4.1 has to be replaced with Theorem \ref{th:inf_dim_lag_multiplier} to apply the abstract theorem of the infinite-dimensional version of the Lagrange multipliers' method.
\end{remark}
\begin{remark}
    The statement of Theorem \ref{th:lagrange_multiplier_nostro} holds for minimizers. Indeed, computing explicit solutions (critical points for $\mathcal{E}_f$) for \eqref{eq:armonica_th} - \eqref{eq:uno_th} - \eqref{eq:due_th} - \eqref{eq:tre_th} - \eqref{eq:quattro_th} is quite hard since it is a system of ODE not in its normal form. Hence, since Theorem \ref{th:esistenza_frame} provides the existence for at least a minimizer for the functional $\mathcal{E}_f$, then necessarily \eqref{eq:armonica_th} - \eqref{eq:uno_th} - \eqref{eq:due_th} - \eqref{eq:tre_th} - \eqref{eq:quattro_th} hold at the minimum.
\end{remark}
Assuming {\em a priori} regularity, we can eliminate the Lagrange multipliers in the system \eqref{eq:armonica_th} - \eqref{eq:uno_th} - \eqref{eq:due_th} - \eqref{eq:tre_th} - \eqref{eq:quattro_th} obtaining the following result.
\begin{theorem}
\label{th:elimino_l_m}
Assume that $f$ is of class $C^3$ and $\Psi$ of class $C^2$. Let $\left((\vect{t}|\vect{n}|\vect{b}), \vect{X}\right) \in \mathcal{C}_f$ be a smooth solution of \eqref{eq:armonica_th} - \eqref{eq:uno_th} - \eqref{eq:due_th} - \eqref{eq:tre_th} - \eqref{eq:quattro_th}.
 Then at any point where $\kappa \neq 0$, it holds
 \begin{empheq}[left=\empheqlbrace]{align}
&\diver\,\left({\rm D}\Psi(\nabla \vect{X})\right) = 0 &&\hbox{on }D,\\
 \label{eq:elimino_uno}
 &\left(\frac{f'_b}{\kappa}\right)'' - \frac{\mu}{2}\left({\rm D}\Psi(\nabla \vect{X}\circ c)\right)\vect{\nu}_D \cdot \vect{b} = 2 f'_a \tau + f_a \tau' -(f_b \kappa)' +\tau^2 \frac{f'_b}{\kappa} &&\hbox{on }(0,2\pi),\\
 \label{eq:elimino_due}
 & \left(-\frac{f_a''}{\kappa} -\frac{2 \tau}{\kappa} \left(\frac{f'_b}{\kappa}\right)' -\frac{f'_b}{\kappa^2}\tau' +\frac{\tau^2 f_a}{\kappa} -f_b \tau -\frac{\mu}{2 \kappa} \left({\rm D}\Psi(\nabla \vect{X}\circ c)\right)\vect{\nu}_D \cdot \vect{n} \right) ' \\\nonumber
 &\qquad \qquad \qquad\qquad \qquad \qquad\qquad\;\, \quad+\frac{\mu}{2} \left({\rm D}\Psi(\nabla \vect{X}\circ c)\right)\vect{\nu}_D \cdot \vect{t} -\kappa f'_a -f'_b \tau = 0&&\hbox{on }(0,2\pi).
\end{empheq}
\end{theorem}
\begin{proof}
The proof is just the calculation to eliminate the Lagrange multipliers. First of all, from \eqref{eq:quattro_th}, in a point in which $\kappa \neq 0$, we have
\[
    \omega = \frac{f'_b}{\kappa}.
\]
Then, differentiating \eqref{eq:due_th} with respect to $\vartheta$, using $\vect{b}' = -\tau \vect{n}$ and the expression of $\vect{\chi}$ given by \eqref{eq:tre_th}, we obtain
$$
\left(\frac{f'_b}{\kappa}\right)'' - \frac{\mu}{2}\left({\rm D}\Psi(\nabla \vect{X}\circ c)\right)\vect{\nu}_D \cdot \vect{b} = 2 f'_a \tau + f_a \tau' -(f_b \kappa)' +\tau^2 \frac{f'_b}{\kappa},
$$
which is exactly \eqref{eq:elimino_uno}. To get the second equation, we differentiate \eqref{eq:uno_th} with respect to $\vartheta$. Using $\vect{n}' = -\kappa \vect{t} + \tau \vect{b}$ and the expression of $\vect{\chi}$ in \eqref{eq:tre_th}, we get
\[
    -f_a'' -2 \left(\frac{f_b'}{\kappa}\right)' \tau -\frac{f'_b}{\kappa} \tau' + \tau^2 f-a -f_b \kappa \tau - \frac{\mu}{2}\left({\rm D}\Psi(\nabla \vect{X}\circ c)\right)\vect{\nu}_D \cdot \vect{n} + \kappa \vect t \cdot \int_0^\vartheta\vect{\lambda}_6\, ds = -\kappa\, \vect{\lambda}\cdot \vect{t}.
\]
Since $\vect{\lambda}$ is a constant vector and we are eliminating the Lagrange multiplier at the points where $\kappa \neq 0$, we can divide everything with respect to $\kappa$. Hence, differentiating another time with respect to $\vartheta$ and using $\vect{t}' = \kappa \vect{n}$, we obtain
$$
 \left(-\frac{f_a''}{\kappa} -\frac{2 \tau}{\kappa} \left(\frac{f'_b}{\kappa}\right)' -\frac{f'_b}{\kappa^2}\tau' +\frac{\tau^2 f_a}{\kappa} -f_b \tau -\frac{\mu}{2 \kappa} \left({\rm D}\Psi(\nabla \vect{X}\circ c)\right)\vect{\nu}_D \cdot \vect{n} \right) ' +\frac{\mu}{2} \left({\rm D}\Psi(\nabla \vect{X}\circ c)\right)\vect{\nu}_D \cdot \vect{t}= \kappa f'_a +f'_b \tau,
$$
which is exactly \eqref{eq:elimino_due} and this concludes the proof.
\end{proof}

\begin{remark}
\label{rem:uguali}
In the case
\[
f(\kappa, \tau):= \kappa^2 \qquad \hbox{ and } \qquad \Psi(\nabla \vect{X}):= \abs{\nabla \vect{X}}^2,
\]
the two introduced approaches, namely the parametrized representation Equation \eqref{eq:solo_superficie} - \eqref{eq:system_tot2} and the framed one \eqref{eq:elimino_uno} - \eqref{eq:elimino_due}, have to coincide. In order to verify this, we assume, for simplicity, everything smooth enough. First of all, \eqref{eq:elimino_uno} - \eqref{eq:elimino_due} for the specified choice of $f$ and $\Psi$ simplify into
\begin{empheq}[left=\empheqlbrace]{align}
 \label{eq:elimino_uno_simple}
 &-\mu \left(\nabla \vect{X} \circ c\right)\vect{\nu}_D\cdot \vect{b} = 4 \kappa' \tau + 2 \kappa \tau',\\
 \label{eq:elimino_due_simple}
 & \left(-\frac{2 \kappa''}{\kappa} + 2 \tau^2 - \frac{\mu}{\kappa} \left(\nabla \vect{X}\circ c\right)\vect{\nu}_D\cdot \vect{n} \right)' +\mu\left(\nabla \vect{X} \circ c\right)\vect{\nu}_D\cdot \vect{t} -2 \kappa \kappa' = 0.
\end{empheq}
To obtain \eqref{eq:elimino_uno_simple} - \eqref{eq:elimino_due_simple}, we can rewrite the parametrized curve $\vect{r}$ in terms of the orthonormal frame $(\vect{t}|\vect{n}|\vect{b})$, i.e.
$$
2 \vect{r}''' = 2 \vect t'' = 2 \left(\kappa \vect{n}\right)' = 2 \kappa' \vect{n} + 2 \kappa \vect{n}' = 2 \kappa' \vect{n} + 2 \kappa \left(-\kappa \vect{t} + \tau \vect{b}\right),
$$
hence considering the vertical component with respect to $\vect r$, we have
\begin{equation}
    \label{eq:r_tre_ort}
    \left(2 \vect{r}'''\right)_{\perp} = 2 \kappa' \vect{n} + 2 \kappa \tau \vect{b}.
\end{equation}
Let the function $\vect{g}$ be defined as
\begin{equation}
    \label{eq:g}
    \vect{g}(\theta):= \int_0^\theta(\nabla \vect{X} \circ c)\vect{\nu}_D\,ds.
\end{equation}
Of course 
\begin{equation}
    \label{eq:g_perp}
    \vect{g}_{\perp} =\prt{\vect{g} \cdot \vect{n}}\vect{n} + \prt{\vect{g} \cdot \vect{b}}\vect{b}.
\end{equation}
To eliminate the constant obtained through the parametrized approach we differentiate \eqref{eq:r_tre_ort} and \eqref{eq:g_perp} with respect to $\vartheta$ and using \eqref{frenet} we get
\[
\begin{aligned}
((2 \vect{r}''')_\perp)' &= \vect{t} \left(- 2 \kappa \kappa'\right) + \vect{n} \left(2 \kappa'' -2 \kappa \tau^2\right) + \vect{b}\left(4 \tau \kappa' + 2 \kappa \tau'\right),\\
-(\vect g _\perp)' &= \vect{t}\left(\kappa \vect{g}\cdot \vect{n}\right) + \vect{n}\left(-\vect{g}'\cdot \vect{n} +\kappa \vect{g} \cdot \vect{t}\right) + \vect{b}\left(-\vect{g}' \cdot\vect{b}\right).
\end{aligned}
\]
Considering the component along $\vect{b}$, we get \eqref{eq:elimino_uno_simple}, indeed
\[
4 \tau \kappa' + 2 \kappa \tau' = -\vect{g}' \cdot\vect{b} \stackrel{\eqref{eq:g}}{=} -\mu \left(\nabla \vect{X} \circ c\right)\vect{\nu}_D\cdot \vect{b}.
\]
Then, considering the projection along $\vect{n}$, differentiating with respect to $\vartheta$ and using $\vect{g}\cdot \vect{n} =-2 \kappa'$, i.e. the component along the tangent vector, we get
\[
\left(-\frac{2 \kappa''}{\kappa} + 2 \tau^2 -\frac{\vect{g'}\cdot \vect{n}}{\kappa}\right)' = -\vect{g}' \cdot \vect{t} + 2 \kappa \kappa',
\]
which is exactly \eqref{eq:elimino_due_simple} using the definition of $\vect{g}'$ given \eqref{eq:g}.

\end{remark}

\section{Numerical approach}\label{sec:third-approach}

In this section, we introduce an approach to be employed specifically in the numerical simulations. It is tailored for the requirements of numerical simulations based upon the finite element method \cite{brenner2008mathematical,ciarlet2002finite,ern2004theory}. Since a typical finite element method hinges upon an integral formulation of the problem, the biggest obstacle in the application of the formulation reported in the parametrized version, introduced in Section \ref{sec:first-approach}, is the pointwise evaluation of constraints \eqref{eq:velocita_1} and \eqref{eq:traccia}. For what concerns the length preserving constraint \eqref{eq:velocita_1}, we will replace it by an integral formulation introducing a Lagrange multiplier to penalize it. Referring to the trace constraint \eqref{eq:traccia}, we will instead identify the curve $\vect{r}$ with the trace of the map $\vect{X}$. This implies that we are going to reformulate the problem and the {\em modified} constraints directly on the single membrane, i.e. on the disc-type parametrization $\vect{X}$.

\subsection{Setting of the problem}
Identifying the curve with the boundary of the disc-parametrization of the elastic membrane, we introduce 
\[
\vect{X} \in W^{3,2}(D;\R^3).
\]
Then $\hat{\vect X} \in W^{\frac{5}{2}, 2}(\partial D;\R^3)$. In order to define the involved constraints, on $\partial D$, we introduce the following tangential directional derivative
\begin{equation*}
\partial_\theta:
\left\{
\begin{aligned}
W^{\frac{5}{2},2}(\partial D; \R^3) &\to W^{\frac{3}{2},2}(\partial D; \R^3),\\
\hat{\vect{X}} &\mapsto\frac{{\rm d}}{{\rm d}\vartheta}\left(\hat{\vect{X}}\circ c(\vartheta)\right)
\end{aligned}
\right.
\end{equation*}
where $\hat{\vect{X}}$ is the trace of $\vect{X}$ and $c(\vartheta) = (\cos \theta, \sin \theta)$.
Similarly $\partial^2_{\theta} \hat{\vect{X}} = \partial_\theta \left(\partial_\theta \hat{\vect{X}}\right)$, obtaining that $\partial^2_{\theta} \hat{\vect{X}} \in W^{ \frac{1}{2},2}(\partial D, \R^3)$.
With these notations, the closure of the curve and the tangent vector, together with the trace constraint are unnecessary, since they are satisfied by definition. Moreover, in the following, all the integrals defined over the boundary of the disc $\partial D$ have to be interpreted as a line integral over the interval $[0, 2\pi]$. Concerning the length preserving constraint \eqref{eq:velocita_1}, which is not {\em a priori} true here, it can equivalently be reformulated as
\begin{equation}
\left|\partial_\theta \hat{\vect{X}}\right| = 1,\quad\text{on}\quad\partial D.
\label{eq:velocita_3}
\end{equation}
Since this approach aims to be applied in numerical experiments, we are going to substitute the clamping condition \eqref{eq:r_0}, which defines the exact position of the curve in the space, with an alternative integral constraint, i.e.
\begin{equation}
\int_D \vect{X} \; du dv = \vect{0}.
\label{eq:clamp_3b}
\end{equation}
Precisely, \eqref{eq:clamp_3b} provides an alternative way to control the variations of constants, avoiding to get two solutions which differ by a constant rotation, i.e. by a rigid motion.

Therefore, the set of constraints is given by
\[
    \label{eq:set_constraint_3}
   \mathcal{C}_{n}= \left\{\vect{X}\in W^{3,2}(D; \R^3): \text{\eqref{eq:velocita_3} \text{ and } \eqref{eq:clamp_3b} hold true}\right\},
\]
where the subscript $n$ refers to the {\em numerical approach} adopted here.
Finally, we introduce the energy functional: let $\mathcal{E}_n \colon \mathcal{C}_{n} \to [0,+\infty)$ be given by
\begin{equation}
    \label{eq:functional_3}
   \mathcal{E}_n[\vect{X}] = \int_{\partial D} \abs{\partial^2_{\theta} \hat{\vect{X}}}^2\, d\theta + \frac{\mu}{2}\int_D \textgoth{C}\nabla \vect{X}:\nabla\vect{X}\, dudv,
\end{equation}
where the reason of choosing $\vect{X} \in W^{3,2}(D;\R^3)$ is apparent: the presence of the term $\partial^2_{\theta} \hat{\vect{X}}$ forces to assume an higher regularity on the entire domain $D$ (rather than its boundary $\partial D$ alone), compared to the weaker regularity assumptions laid out in Sections \ref{sec:first-approach} and \ref{sec:second-approach}. This aspect will also have important implications in the choice of an appropriate finite element method, as we will discuss in Section \ref{sec:numerical-results}.

\subsection{First-order necessary conditions}

To derive the first-order necessary conditions, we slightly modify the energy functional to take into account the length preserving constraint \eqref{eq:velocita_3}. Indeed, let $\widetilde{\mathcal{E}}_n \colon  \widetilde{\mathcal{C}}_n \to \R$ be given by
\begin{equation}
    \label{eq:F_ell}
    \widetilde{\mathcal{E}}_n[\vect{X}, \ell]= \int_{\partial D} \abs{\partial^2_{\theta} \hat{\vect{X}}}^2\, d\vartheta + \frac{\mu}{2}\int_D \textgoth{C}\nabla \vect{X}:\nabla\vect{X}\, dudv + \int_{\partial D} \ell \; \left[\left|\partial_\theta \hat{\vect{X}}\right|^2 - 1\right] d\vartheta,
\end{equation}
where
\[
    \widetilde{\mathcal{C}}_n= \left\{(\vect{X}, \ell) \in W^{3,2}(D; \R^3) \times W^{\frac{1}{2},2}(\partial D;\R): \eqref{eq:clamp_3b} \hbox{ holds}\right\}.
\]
Roughly speaking, we are going to penalize the square of \eqref{eq:velocita_3} through an infinite-dimensional Lagrange multiplier defined on $\partial D$.

To derive the first-order necessary conditions, we consider the following admissible variations for both the membrane $\vect{X}$ and the Lagrange multiplier $\ell$, i.e.
\begin{align}
    \label{eq:variation_X_3}
    &\vect{X}_t(u,v):= \vect{X}(u,v) + t\ \vect{V}(u,v), &&\forall\, \vect{V} \in C^\infty(\overline{D}; \R^3) \hbox{  and  } \int_D \vect{V} \; du dv = \vect{0},\\
    \label{eq:variation_ell_3}
    &\ell_s(\vartheta):= \ell(\vartheta) + s\ m(\vartheta) &&\forall\, m \in C^\infty(\overline{\partial D}; \R).
\end{align}
We notice that the introduced variations for the two unknowns $(\vect{X}, \ell)$ are independent in this formulation, which makes exhibiting an admissible variation considerably easier than the one used in Section \ref{sec:first-approach}. Now, we compute
\begin{equation*}
\frac{{\rm d}}{{\rm d t}} \widetilde{\mathcal{E}}_n[\vect{X}_t, \ell_s]_{\bigr|_{\substack{(t, s) = \vect{0}}}} \qquad \text{and} \qquad \frac{{\rm d}}{{\rm d s}} \widetilde{\mathcal{E}}_n[\vect{X}_t, \ell_s]_{\bigr|_{\substack{(t, s) = \vect{0}}}}.
\end{equation*}

Let us start from the partial derivative w.r.t. $t$. We get
\begin{equation*}
\begin{aligned}
\frac{{\rm d}}{{\rm d t}} \widetilde{\mathcal{E}}_n\left[\vect{X}_t, \ell_s\right]_{\bigr|_{(t, s) = \vect{0}}} = \frac{{\rm d}}{{\rm d t}}\left\{\underbrace{\int_{\partial D} \abs{\partial^2_{\theta} \hat{\vect{X}}_t}^2\, d\theta}_{\widetilde{\mathcal I_1}} + \underbrace{\frac{\mu}{2} \int_D \textgoth{C}\nabla \vect{X}_t:\nabla \vect{X}_t\, du dv}_{\widetilde{\mathcal I_2}} + \underbrace{\int_{\partial D} \ell_s \; \left[\left|\partial_\theta \hat{\vect{X}}_t\right|^2 - 1\right] d\theta}_{\widetilde{\mathcal I_3}}\right\}_{\biggr|_{\substack{(t, s) = \vect{0}}}},
\end{aligned}
\end{equation*}
and we study each term separately. Concerning $\widetilde{\mathcal{I}_1}$, we get
\begin{equation*}
\frac{{\rm d}}{{\rm dt}}{\widetilde{\mathcal I_1}}_{\big|_{(t, s) = \vect 0}}= \int_{\partial D} \frac{{\rm d}}{{\rm d t}}  \partial^2_{\theta} \hat{\vect{X}}_t \cdot \partial^2_{\theta} \hat{\vect{X}}_t \, d\theta_{\bigr|_{\substack{(t, s) = \vect{0}}}} = 2 \int_{\partial D} \partial^2_{\theta} \hat{\vect{V}} \cdot \partial^2_{\theta} \hat{\vect{X}}\, d\theta.
\end{equation*}
For the membrane contribution $\widetilde{\mathcal{I}}_2$, we obtain
\begin{equation*}
\frac{{\rm d}}{{\rm d t}}{\widetilde{\mathcal I_2}}_{\bigr|_{\substack{(t, s) = \vect{0}}} }= \mu \int_D \textgoth{C}\nabla \vect{X}: \nabla \vect{V}\, dudv.
\end{equation*}
Finally, the last term simplifies into
\begin{equation*}
\frac{{\rm d}}{{\rm d t}}{\widetilde{\mathcal I_3}}_{\bigr|_{\substack{(t, s) = \vect{0}}}} = \int_{\partial D} \frac{{\rm d}}{{\rm d t}} \left(\ell_s \; \left[\left|\partial_\theta \hat{\vect{X}}_t\right|^2 - 1\right]\right)_{\bigr|_{\substack{(t, s) = \vect{0}}}}\, d\theta
= \int_{\partial D} \ell \; \frac{{\rm d}}{{\rm d t}} \left[\left|\partial_\theta \hat{\vect{X}}_t\right|^2\right]_{\bigr|_{\substack{t = 0}}} d\theta = 2 \int_{\partial D} \ell \; \partial_\theta \hat{\vect{V}} \cdot \partial_\theta \hat{\vect{X}} \; d\theta.
\end{equation*}

Moving to the partial derivative w.r.t. $s$, the only non-zero term in the differentiation is the last one of Equation \eqref{eq:F_ell}, since the other ones do not depend on $s$. Hence, we get
\begin{equation*}
\begin{aligned}
\frac{{\rm d}}{{\rm d s}} \widetilde{\mathcal{E}}_n\left[\vect{X}_t, \ell_s\right]_{\bigr|_{\substack{(t, s) = \vect{0}}}} = \frac{{\rm d}}{{\rm d s}} \int_{\partial D} \ell_s \; \left[\left|\partial_\theta \hat{\vect{X}}_t\right|^2 - 1\right] d\theta_{\bigr|_{\substack{(t, s) = \vect{0}}}} = \int_{\partial D} m \; \left[\left|\partial_\theta \hat{\vect{X}}\right|^2 - 1\right] d\theta.
\end{aligned}
\end{equation*}

Collecting all the terms and imposing both partial derivatives to be zero to seek critical points, we end up with
\begin{equation}
    \label{eq:system_tot1}
    \left\{
\begin{aligned}
&2 \int_{\partial D} \partial^2_{\theta} \hat{\vect{V}} \cdot \partial^2_{\theta} \hat{\vect{X}}\, d\theta + \mu \int_D \textgoth{C}\nabla \vect{X}: \nabla \vect{V}\, dudv + 2 \int_{\partial D} \ell \; \partial_\theta \hat{\vect{V}} \cdot \partial_\theta \hat{\vect{X}} \; d\theta = 0,\\
&\int_{\partial D} m \; \left[\left|\partial_\theta \hat{\vect{X}}\right|^2 - 1\right] d\theta= 0,
\end{aligned}
\right.
\end{equation}
which, in the following, we will refer to them as the first-order necessary conditions and we are going to study numerically in the next section.

\begin{remark}[Differences with Theorem \ref{thm:system_tot}]
Even though the techniques employed to derive the first-order necessary conditions are quite similar to the ones adopted in the parametrized approach introduced in Section \ref{sec:first-approach}, there are several aspects in the computations to derive \eqref{eq:system_tot1} that favor their applicability for numerical simulations.
\begin{itemize}
\item The definition of the variation for the curve (and thus, for the membrane) in Theorem \ref{thm:system_tot} requires a pointwise orthogonality property between the arbitrary vector field $\vect{\eta}$ and the curve $\vect{r}$ (see Lemma \ref{lemma:variazione_curva}) to satisfy the length preserving constraint \eqref{eq:velocita_1}. However, imposing such a request in a finite element sense is not easy and elementary. In contrast, here the variation $\vect{X}_t$ for the membrane $\vect{X}$ does not require such a property (see \eqref{eq:variation_X_3}).
\item The optimality conditions \eqref{eq:system_tot1} are just reported in their integral formulation. Indeed, differently from Theorem \ref{thm:system_tot} where the strong form is derived thanks to the gained regularity, in the numerical formulation there is no need to obtain the strong form since finite element methods only require the integral formulation, rather than the strong one.
\item In the parametrized approach, the {\em explicit} dependence of the Lagrange multiplier $\lambda \in \mathbb{R}^+$ drops out from the first-order necessary conditions in Theorem \ref{thm:system_tot} because of \eqref{eq:traccia}. Differently, in this setting, the infinite-dimensional Lagrange multiplier $\ell \in W^{\frac{1}{2},2}(\partial D;\R)$ appears in \eqref{eq:system_tot1}, implying that our numerical scheme will seek both the disc-type parametrization of the membrane $\vect{X}$ and the Lagrange multiplier $\ell$. Precisely, in the following Section \ref{sec:numerical-results}, we will report only results for  $\vect{X}$, since the visualization of $\ell$ hardly offers any physically relevant insights.
\end{itemize}
\end{remark}

\begin{remark}[On the penalization of the square of \eqref{eq:velocita_3} in the definition of $\widetilde{\mathcal{E}}_n$]

To impose the length preserving constraint \eqref{eq:velocita_3}, in the definition of the energy functional $\widetilde{\mathcal{E}}_n$ in \eqref{eq:F_ell} we decided to penalize the square of the velocity, i.e. $\left|\partial_\theta \hat{\vect{X}}\right|^2 - 1$. Such a choice results in getting a specific type of nonlinearity in the first-order necessary conditions:  optimality conditions are quadratically nonlinear. A more natural choice might be the penalization of just  $\left|\partial_\theta \hat{\vect{X}}\right| - 1$, i.e.
\begin{equation}
    \label{eq:derivata_frazione}
        \int_{\partial D} \ell \; \left[\left|\partial_\theta \hat{\vect{X}}\right| - 1\right] d\theta,
\end{equation}
which is nothing but the weak formulation though a Lagrange multiplier of the pointwise constraint defined in \eqref{eq:velocita_3}.
However, substituting $\widetilde{\mathcal{I}}_3$ with \eqref{eq:derivata_frazione} and performing the derivative with respect to $t$, as in the computations to get \eqref{eq:system_tot1}, results in a factor which looks like
$$
\frac{1}{\left|\partial_\theta \hat{\vect{X}}\right|},
$$
i.e. we get a nonlinear problem with fractional terms. Such a rational term may turn out to be more challenging than a quadratic one, because while solving the problem by means of iterative methods reported in Section \ref{sec:numerical-results}, there might exists a point $\tilde{\vartheta} \in \partial D$ such that $\left|\partial_\theta \hat{\vect{X}}(\tilde{\vartheta})\right|= 0$, causing a breakdown of the iterative scheme due to a division by zero.

\end{remark}

\section{Numerical results}\label{sec:numerical-results}

\subsection{Discretization of the problem}
\label{sec:subsec_discretization}
The approach outlined in Section \ref{sec:third-approach} is chosen to obtain computational results based on a finite element discretization. Towards this goal, the disc $D$ is partitioned into a triangulation (mesh) of non overlapping triangular cells. We will refer to such triangulation as $D_\delta$ in the following. The triangulation is generated through \texttt{gmsh} \cite{geuzaine2009gmsh}, and contains 3219 nodes and 6438 cells.
The mesh is unstructured, and is refined close to $\partial D$, i.e. to the elastic curve, ensuring that the cell size close to the boundary curve is approximately half of the mesh size at the origin of the disc, see Figure \ref{fig:mesh}.
\begin{figure}
\centering
\includegraphics[scale=0.27]{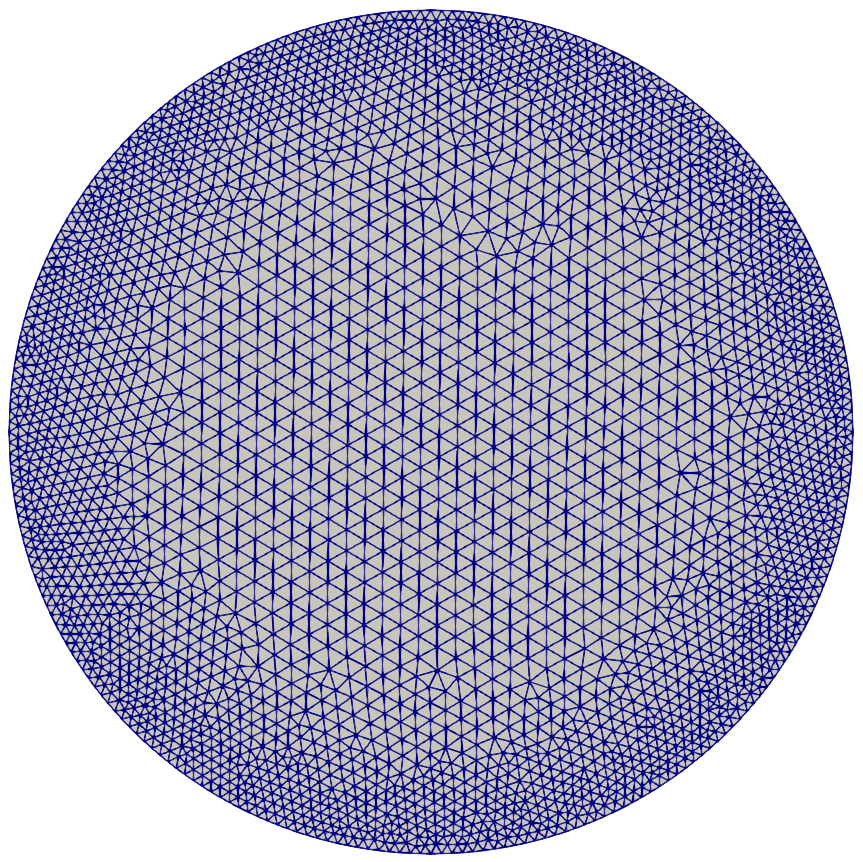}
\caption{Triangulation $D_\delta$ of the disc $D$, showing a mesh refinement on the boundary.}\label{fig:mesh}
\end{figure}

An approximation $\vect{X}_\delta$ of the membrane $\vect{X}$ is sought in a finite dimensional space $\mathcal{X}_\delta$ based on discontinuous Lagrange finite elements of degree 2 on triangles, i.e. the restriction of each component of $\vect{X}_\delta$ to any cell $T \in D_\delta$ is a second order polynomial.
The finite element space $\mathcal{X}_\delta$ thus replaces the infinite dimensional space $W^{3,2}(D; \R^3)$ in the formulation reported in Section \ref{sec:third-approach}.
Concerning the approximation $\ell_\delta$ of the Lagrange multiplier $\ell$ in a finite dimensional space $\mathcal{L}_\delta$, we consider continuous Lagrange finite elements of degree 2 on segments. Such a finite dimensional approximation of $W^{\frac{1}{2},2}(\partial D;\R)$ contains functions which are globally continuous and such that any element of $\mathcal{L}_\delta$ restricted to any boundary face of $D_\delta$ (i.e., any segment in the polygonal approximation $\partial D_\delta$ of $\partial D$) is a second order polynomial. 
Therefore, there are two main differences between the finite dimensional spaces $\mathcal{X}_\delta$ and $\mathcal{L}_\delta$:
\begin{enumerate}
    \item $\mathcal{X}_\delta$ is defined on the entire mesh $D_\delta$, while $\mathcal{L}_\delta$ is only defined over $\partial D_\delta$;
    \item $\mathcal{X}_\delta$ is a space of local polynomials which are globally (possibly) discontinuous on $D_\delta$, while $\mathcal{L}_\delta$ is composed of local polynomials which are also globally continuous on $\partial D_\delta$.
    \end{enumerate}
Concerning the first consideration, it does not need to be justified any further, since it naturally stems from the corresponding infinite dimensional spaces $W^{3,2}(D; \R^3)$ and $W^{\frac{1}{2},2}(\partial D;\R)$, respectively. Regarding the choice of discontinuous elements for $\mathcal{X}_\delta$ in the second comment, it is strongly related to practical implementation limitations. Indeed, an approximation of $W^{3,2}(D; \R^3)$ by means of continuous finite elements would actually require $C^2$ continuity across cells, rather than the standard $C^0$ continuity of Lagrange elements. Finite elements with high-order continuity are certainly less commonly studied than their corresponding standard counterparts; we mention here the Argyris triangle \cite{argyris1968tuba} or the Hsieh-Clough-Toucher triangle \cite{clough1965finite} as notable cases of elements with $C^1$ continuity. However, the implementation of finite elements with high-order continuity is not readily available in the \texttt{FEniCS} project \cite{logg2012automated}, the software we use for our numerical simulations together with the \texttt{multiphenics} \cite{multiphenics} library, which easily allows to restrict the approximation $\ell_\delta$ of the Lagrange multiplier to just the boundary $\partial D_\delta$.

The finite element problem to be solved is: find $(\vect{X}_\delta, \ell_\delta) \in \mathcal{X}_\delta \times \mathcal{L}_\delta$ such that

\begin{equation}\label{eq:system_tot_3_fem}
\left\{
\begin{aligned}
&\mu \int_{D_\delta} \textgoth{C}\nabla \vect{X}_\delta: \nabla \vect{V}_\delta\, dudv
+ 2 \int_{\partial D_\delta} \partial^2_{\theta} \hat{\vect{V}}_\delta \cdot \partial^2_{\theta} \hat{\vect{X}}_\delta\, d\theta
+ 2 \int_{\partial D_\delta} \ell_\delta \; \partial_\theta \hat{\vect{V}}_\delta \cdot \partial_\theta \hat{\vect{X}}_\delta \; d\theta\\
&\qquad \underbrace{- \sum_{e \in E_\delta} \mu \int_e \left\{\!\!\left\{\textgoth{C}\nabla \vect{X}_\delta\right\}\!\!\right\}_e : \left[\!\left[\vect{V}_\delta\right]\!\right]_e\, de
- \sum_{e \in E_\delta} \mu \int_e \left\{\!\!\left\{\textgoth{C}\nabla \vect{V}_\delta\right\}\!\!\right\}_e : \left[\!\left[\vect{X}_\delta\right]\!\right]_e\, de
+ \sum_{e \in E_\delta} \frac{\alpha}{h_e} \int_e \left[\!\left[\vect{X}_\delta\right]\!\right]_e \cdot \left[\!\left[\vect{V}_\delta\right]\!\right]_e \, de}_{\widetilde{\mathcal{D}}}\\
&\qquad \underbrace{+ \varepsilon \int_{D_\delta} \vect{X}_\delta \cdot \vect{V}_\delta \; du dv}_{\widetilde{\mathcal{A}}} = 0,\\
&\int_{\partial D_\delta} m_\delta \; \left[\left|\partial_\theta \hat{\vect{X}}_\delta\right|^2 - 1\right] d\theta = 0
\end{aligned}
\right.
\end{equation}
for every $(\vect{V}_\delta, m_\delta) \in \mathcal{X}_\delta \times \mathcal{L}_\delta$. Precisely, \eqref{eq:system_tot_3_fem} is obtained from \eqref{eq:system_tot1} by a standard Galerkin method on the finite dimensional space $\mathcal{X}_\delta \times \mathcal{L}_\delta$, with the addition of two new terms, denoted by $\widetilde{\mathcal{D}}$ and $\widetilde{\mathcal{A}}$.

The additional term $\widetilde{\mathcal{D}}$ appears due to the choice of discontinuous finite elements for $\mathcal{X}_\delta$, getting that the numerical scheme is a symmetric interior penalty discontinuous Galerkin method, see e.g. \cite{di2011mathematical,hesthaven2007nodal,riviere2008discontinuous}. Here, $E_\delta$ is the set of all the interior edges in $D_\delta$ (i.e., the union of all segments $e \in \partial T$ for some $T \in D_\delta$, excluding the boundary edges $e \in \partial D_\delta$), $h_e$ is the length of an edge $e \in E_\delta$. 
For a fixed $e \in E_\delta$, let $T^+$ and $T^-$ be the two cells which share $e$ (i.e., $e = \overline{T^+} \cap \overline{T^-}$), and let $\mathbf{n}$ denote the outer normal to $e$ from $T^+$; then, the average $\left\{\!\!\left\{ \cdot \right\}\!\!\right\}_e$ and jump $\left[\!\left[ \cdot \right]\!\right]_e$ operators appearing in $\widetilde{\mathcal{D}}$ are defined as
\begin{align*}
\left(\left\{\!\!\left\{\textgoth{C}\right\}\!\!\right\}_e\right)_{ij} &= \frac{1}{2}\left({\textgoth{C}_{ij}}_{|_{\overline{T^+} \cap e}} + {\textgoth{C}_{ij}}_{|_{\overline{T^-} \cap e}}\right), \\
\left(\left[\!\left[  \vect{X} \right]\!\right]_{e}\right)_{ij} &= (\vect{X}_i|_{\overline{T^+} \cap e} + \vect{X}_i|_{\overline{T^-} \cap e}) \mathbf{n}_j.\\
\end{align*}
Finally, $\alpha$ is a penalty coefficient (which we choose equal to $10^4$) which appears in the third addend in $\widetilde{\mathcal{D}}$ to penalize jumps of the solution across neighboring elements. One could also add further similar terms that penalize jumps of first and/or second order derivatives; such terms are often omitted (see e.g. \cite{riviere2008discontinuous}), since, with the proper choice of $\alpha$, our experience in the preparation of the results presented below is that a penalization on the jump of the polynomial solution results in practice in small jumps on its derivatives too.

The additional term $\widetilde{\mathcal{A}}$, instead, represents a penalization of the constraint \eqref{eq:clamp_3b}, since such a constraint is not enforced in the finite element space $\mathcal{X}_\delta$. In producing the numerical results which are discussed afterwards, we have experimentally observed that a small value $\varepsilon = 10^{-5}$ is sufficient to prevent rigid body motions.

\begin{remark}[On polynomial degree equal to 2]
We have seen in Section \ref{sec:third-approach} that the presence of the term
$$
2 \int_{\partial D} \partial^2_{\theta} \hat{\vect{V}} \cdot \partial^2_{\theta} \hat{\vect{X}}\, d\vartheta
$$
in \eqref{eq:system_tot1} forces us to assume $\vect{X} \in W^{3,2}(D;\R^3)$. As discussed above, this has profound numerical implications: we are forced to employ a discontinuous Galerkin method, primarily due to the unavailability of $C^2$ finite elements. Moreover, such a term implies the inability to use lowest order finite elements, i.e. polynomials of degree equal to 1. Indeed, $\partial^2_{\theta} \hat{\vect{V}}_\delta$ and $\partial^2_{\theta} \hat{\vect{X}}_\delta$ would otherwise be identically zero, and thus \eqref{eq:system_tot_3_fem} would ignore the boundary contribution to the energy functional.
\end{remark}

The discrete first-order necessary conditions \eqref{eq:system_tot_3_fem} result in a finite-dimensional quadratically nonlinear problem of total dimension 38355 degrees of freedom, of which 37851 for $\vect{X}_\delta$ and 504 unknowns for $\ell_\delta$. We solve such a nonlinear problem with the \texttt{SNES} solver, part of the \texttt{PETSc} library \cite{petsc-user-ref}, with the following choices:
\begin{itemize}
\item a backtracking procedure in the selection of the step length at every nonlinear iteration; in particular, SNES employs a cubic backtracking \cite{petsc-user-ref}. This procedure is beneficial in ensuring convergence of the nonlinear solver, as we have observed that simpler options (e.g., constant step length) often lead to divergence;
\item rescaling of the initial guess $\vect{X}_{guess}$ so that its perimeter is $2\pi$. This choice ``helps'' the numerical code to fulfil the nonlinear length preserving constraint. Indeed, we observed that providing an initial guess with considerably smaller or larger perimeters often resulted in stagnation to a nonconverged solution with the failure of the length preserving constraint, or even divergence.
\end{itemize}

The stopping criteria for the nonlinear iteration is realized by setting a suitable tolerance $\left(\hbox{approximately } <10^{-6}\right)$ under which we can deduce that the residual norm is very small, i.e. the discrete variational formulation is satisfied. Moreover, {\em a posteriori}, we check the constraints are fulfilled: the length of the solution $\vect{X}_{sol}$ at convergence must be close to $2 \pi$ and the absolute value of the velocity $\abs{\vect{r}'}$ must be close to one. 

We present next the numerical results obtained for two different choices for the linear elasticity tensor $\textgoth{C}$.

\subsection{Identity tensor}
\label{sec:numerical_C_identity}
First of all, we consider the case where the elastic tensor $\textgoth{C}$ coincides with the identity. 
In this section we choose different initial guesses $\vect{X}_{guess}$ for the parametrization, properly {\em rescaled} as discussed above, and we solve the nonlinear weak formulation as presented in Section \ref{sec:subsec_discretization}. All the guesses are written in polar coordinates choosing $\rho \in [0,1]$ to be the radial coordinate, while $\vartheta \in [0, 2\pi]$ is the azimuth angular coordinate. 

We consider as first guess datum $\vect{X}_{guess}$ a disc of radius one, whose formulation in polar coordinates is
$$\vect{X}_{guess}=(-\rho \sin \vartheta, \rho \cos \vartheta,0).
$$
We expect not to be any variations since the disc is a critical point of the functional $\mathcal{E}_n$ defined in \eqref{eq:functional_3}, i.e. we are starting ``close enough'' to a stationary solution. Indeed, the numerical solution $\vect{X}_{sol}$ is exactly the same disc, in the same plane, see Figure \ref{fig:cerchio_guess}. 
\begin{figure}[h!]
\centering
\includegraphics[scale=0.3]{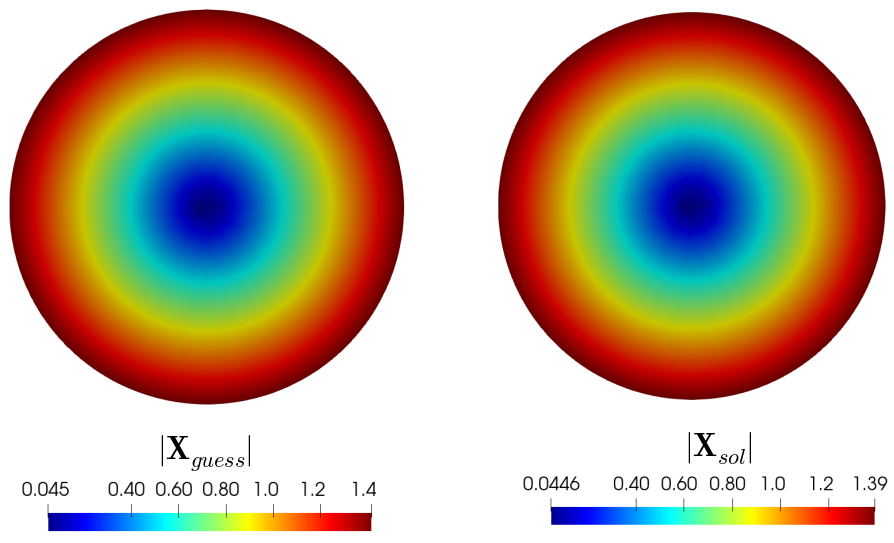}
\caption{Magnitude of the initial guess $\vect{X}_{guess}=(-\rho \sin\vartheta, \rho \cos \vartheta ,0)$ (left), and of the corresponding obtained solution $\vect{X}_{sol}$ (right), which is in this case a circle.}\label{fig:cerchio_guess}
\end{figure}

We consider next
$$
\vect{X}_{guess} = \left(\rho\sin\vartheta, \rho\cos\vartheta, \rho \sin\vartheta\right)
$$
as the second initial guess, which is the polar representation of an ellipse. Again in this case, see Figure \ref{fig:ellisse_guess}, the numerical minimization converges to the disc in the same plane of the initial ellipse 
confirming that nonlinear iterations converge to a stable critical point when starting close enough to it. 
\begin{figure}[h!]
\centering
\includegraphics[scale=0.25]{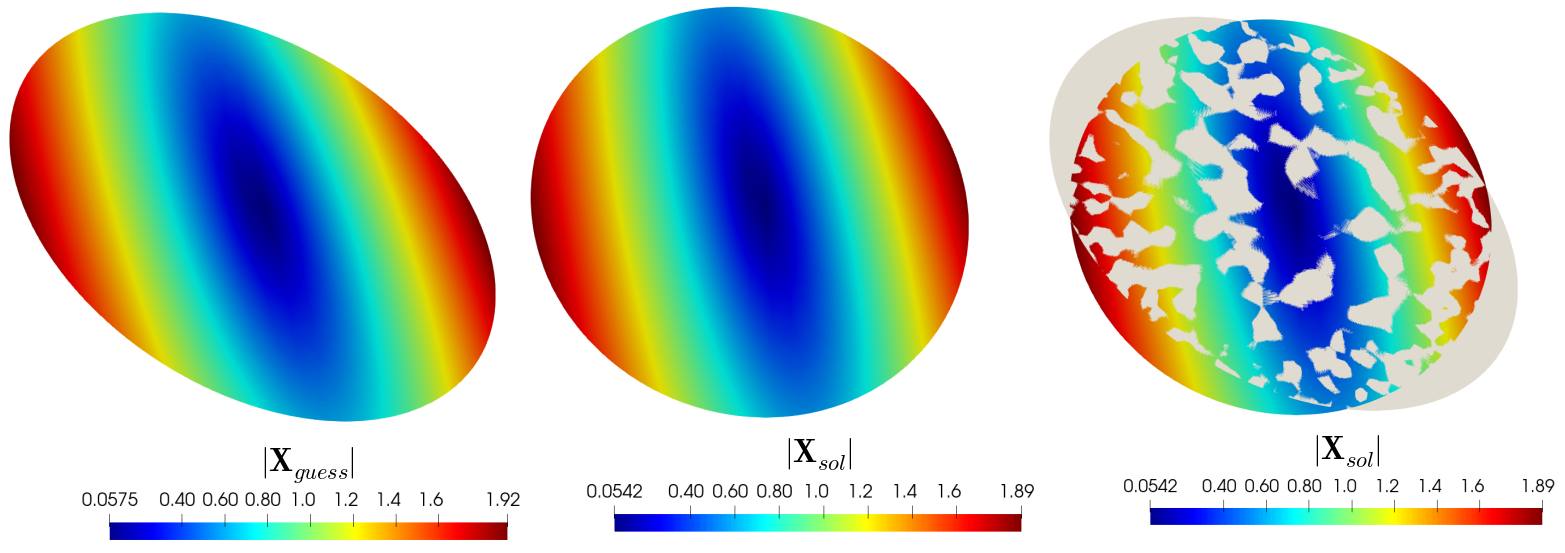}
\caption{Magnitude of the initial guess $\vect{X}_{guess}=(\rho \sin\vartheta, \rho \cos\vartheta,\rho \sin\vartheta)$ (left), and of the corresponding obtained solution $\vect{X}_{sol}$ (center), which in this case is a circle. On the right, we plot the initial guess (in gray) and the obtained solution (colored) to show that they lay in the same plane. }\label{fig:ellisse_guess}
\end{figure}

Finally, we prescribe a non-planar initial guess: the paraboloid in Figure \ref{fig:paraboloide_guess} is given by
$$\vect{X}_{guess} = \left(\rho \cos\vartheta, \rho \sin\vartheta, - \rho^2\right).
$$ 
Running the numerical code, we obtain as a critical point the disc in the plane passing through the maximum point $(0,0,0)$ and perpendicular to the axis of the paraboloid, that is the solution is progressively flattened on such a plane during the nonlinear iterations. Indeed, due to a change of sign in the curvature, any other configuration would cost more in terms of energy.
\begin{figure}[h!]
\centering
\includegraphics[scale=0.25]{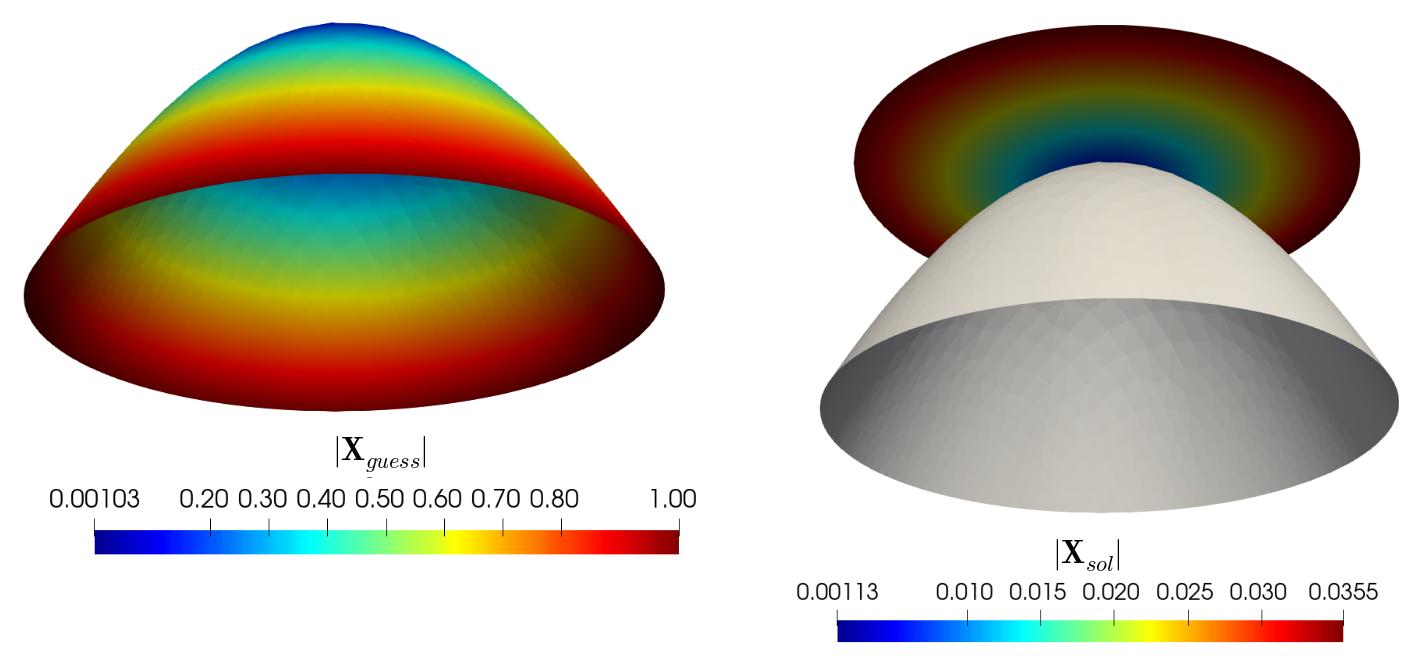}
\caption{Magnitude of the initial guess $\vect{X}_{guess}=(\rho\cos\vartheta,\rho \sin\vartheta,-\rho^2)$ (left), and of the corresponding obtained solution $\vect{X}_{sol}$ overlaid with the initial guess in grey (right) to show the plane which contains $\vect{X}_{sol}$.}\label{fig:paraboloide_guess}
\end{figure}

To try to break the circular symmetry, we consider as initial guesses two {\em ad hoc} functions, i.e. in Figure \ref{fig:patatina_guess} the initial guess is chosen to be $$
\vect{X}_{guess} = \left(\rho\cos\vartheta, - \rho\sin\vartheta, \rho^2 \sin^2\vartheta\right),
$$ 
while, to plot Figure \ref{fig:calzascarpe_guess}, we consider 
$$\vect{X}_{guess}=\left(\rho\cos\vartheta, - \rho\sin\vartheta, \rho^2 \sin^2\vartheta+ \sin\left(\frac{\pi}{2}\rho\right)\cos\vartheta\right).
$$
In both cases, a critical point is the disc. The interesting aspect is that, if in the previous results the lying plane of the disc can be deduced just by simple geometrical considerations, here the symmetry is broken. Indeed, we notice that increasing/decreasing the shear modulus $\mu$, such a plane changes. The solutions reported in Figures \ref{fig:patatina_guess}-\ref{fig:calzascarpe_guess} correspond to the choice $\mu = 1$.
\begin{figure}[h!]
\centering
\includegraphics[scale=0.25]{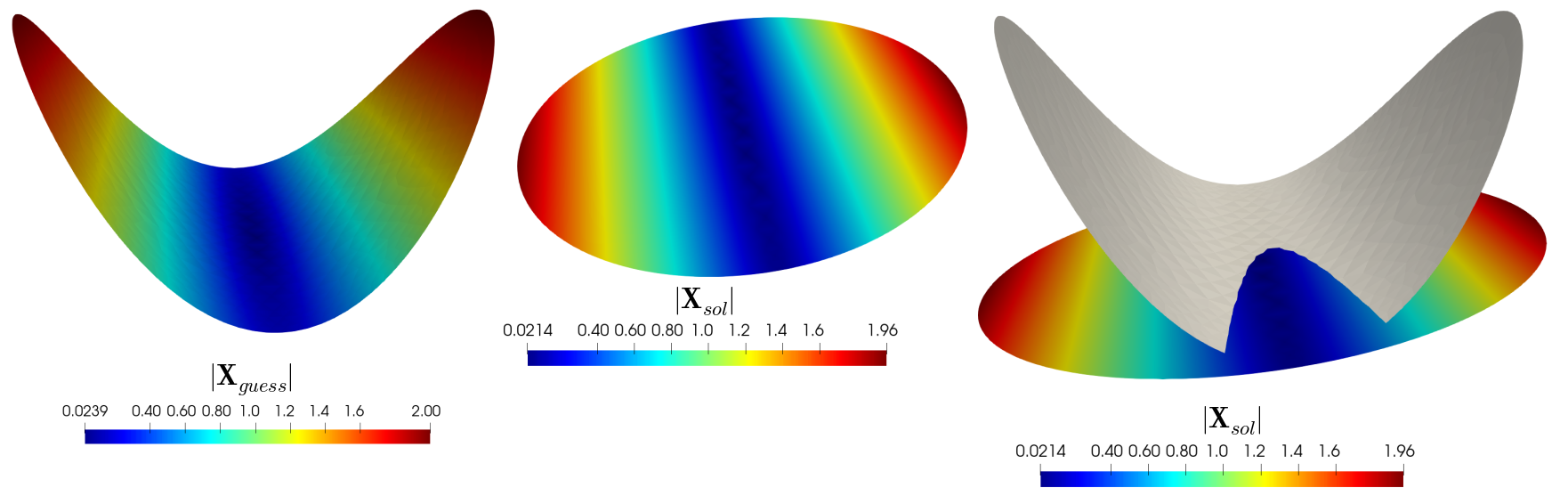}
\caption{Magnitude of the initial guess $\vect{X}_{guess}=(\rho\cos\vartheta, - \rho\sin\vartheta, \rho^2 \sin^2\vartheta)$ (left), and of the corresponding obtained solution $\vect{X}_{sol}$ (center), which in this case is a circle. On the right, we plot the initial guess (in gray) and the obtained solution (colored) to identify its position in $\mathbb{R}^3$ with respect to the initial guess.}\label{fig:patatina_guess}
\end{figure}
\begin{figure}[h!]
\centering
\includegraphics[scale=0.25]{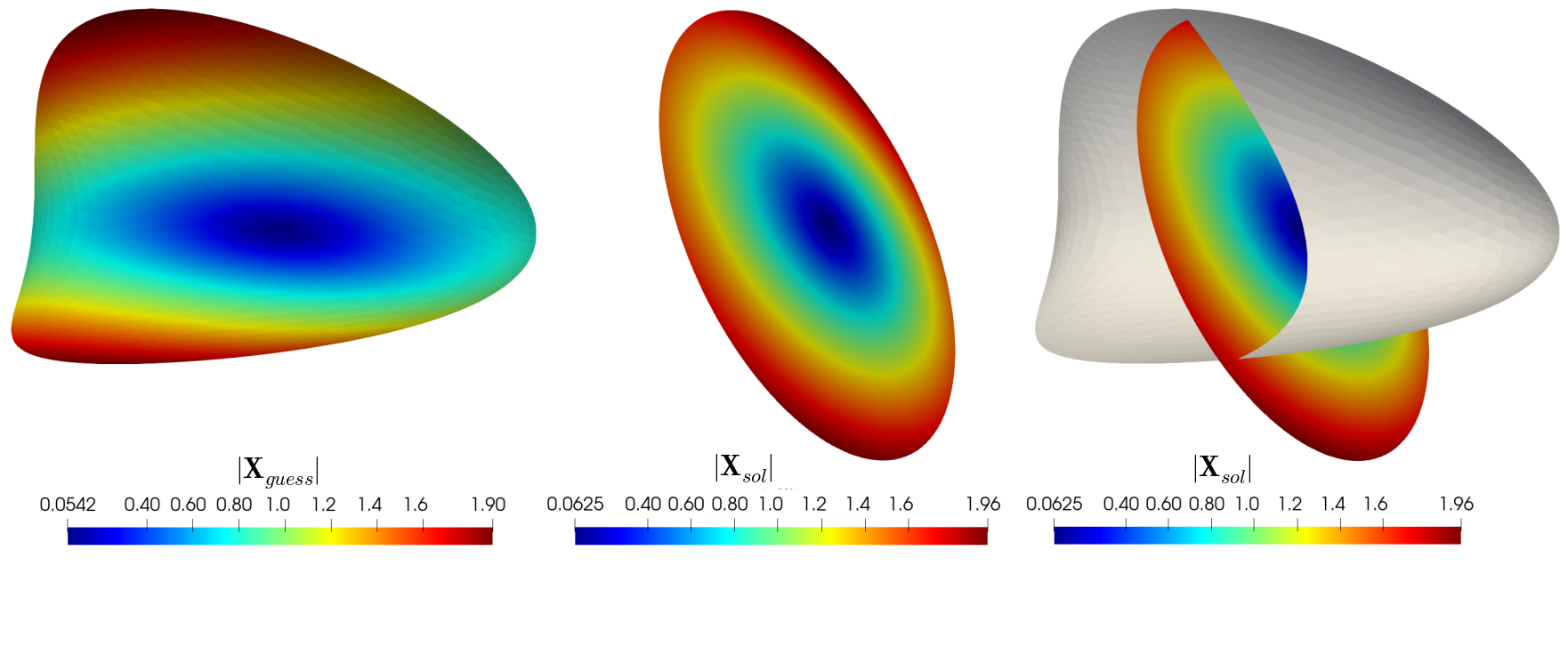}
\caption{Magnitude of the initial guess $\vect{X}_{guess}=\left(\rho\cos\vartheta,\rho\sin\vartheta,-\rho^2\right)$ corresponding obtained solution $\vect{X}_{sol}$ (center), which in this case is a circle. On the right, we plot the initial guess (in gray) and the obtained solution (colored) to identify its position in $\mathbb{R}^3$ with respect to the initial guess.}\label{fig:calzascarpe_guess}
\end{figure}

Unfortunately, considering $\textgoth{C}$ as the identity tensor, we are not able to obtain critical points different from the disc. This motivates us to consider next an elasticity tensor $\textgoth{C}$ different from the identity.

\subsection{Elasticity tensor different from the Identity Matrix}
In this section, we choose different expressions for the linear elastic tensor $\textgoth{C}$, different from the identity. In particular, in Figure \ref{fig:rottura_sym_guess_C_2} its expression is 
$$
\textgoth{C}_{ijk\ell} = \delta_{ij} \delta_{k\ell} + 2 \left(1 + 10 \sin^2(\rho\cos\vartheta)\right)\delta_{ik}\delta_{jl},
$$
while in Figure \ref{fig:rottura_sym_guess_C_1} it is given by
$$
\textgoth{C}_{ijk\ell} = \abs{\rho \cos\vartheta} \delta_{ij} \delta_{k\ell} + 2 \delta_{ik}\delta_{jl}.
$$

The same initial condition $\vect{X}_{guess}$
$$
\vect{X}_{guess}=\left(\rho\cos\vartheta, - \rho\sin\vartheta, \rho \sin\vartheta+ \sin\left(\frac{\pi}{2}\rho\right)\cos(2\vartheta)\right),
$$
is set in both cases.

We run the numerical simulations and we find, in both cases, critical points different from the disc. Precisely, the critical points present self-intersections. As discussed in the introduction, this is something that should be expected in minimizing the Dirichlet energy functional for the area contribution (see Chapter 8 of \cite{morgan2016geometric}) and thus we were set out to show in this work.
The obtained solutions look profoundly different due to anisotropy effects in $\textgoth{C}$. Indeed, such choices of the linear elastic tensor favor different directions at which the stress will be maximum, thus resulting in two completely dissimilar solutions: a planar one with self-intersections in Figure \ref{fig:rottura_sym_guess_C_2} and a non-planar one with self-intersections in Figure \ref{fig:rottura_sym_guess_C_1}. 
\begin{figure}[h!]
\centering
\includegraphics[scale=0.35]{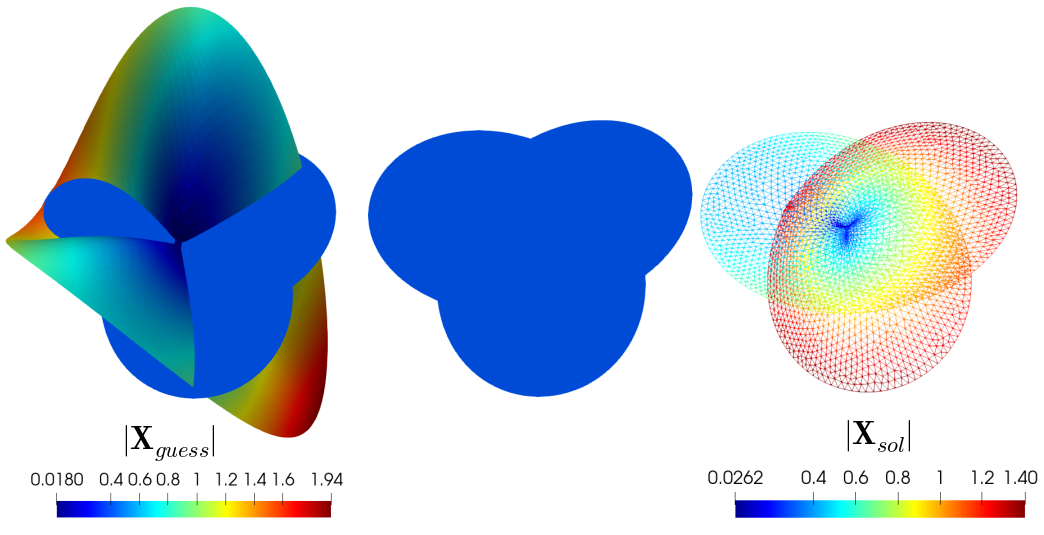}
\caption{On the left, magnitude of the initial guess $\vect{X}_{guess}=\left(\rho\cos\vartheta, - \rho\sin\vartheta, \rho \sin\vartheta+ \sin\left(\frac{\pi}{2}\rho\right)\cos(2\vartheta)\right)$ with $\textgoth{C}_{ijk\ell} = \abs{\rho \cos\vartheta} \delta_{ij} \delta_{k\ell} + 2 \delta_{ik}\delta_{jl}$ (colored), overlaid with the corresponding obtained solution $\vect{X}_{sol}$ in blue. We plot next the shape of $\vect{X}_{sol}$ (center) and its magnitude (right), resulting in planar self-intersections. }\label{fig:rottura_sym_guess_C_2}
\end{figure}
\begin{figure}[h!]
\centering
\includegraphics[scale=0.35]{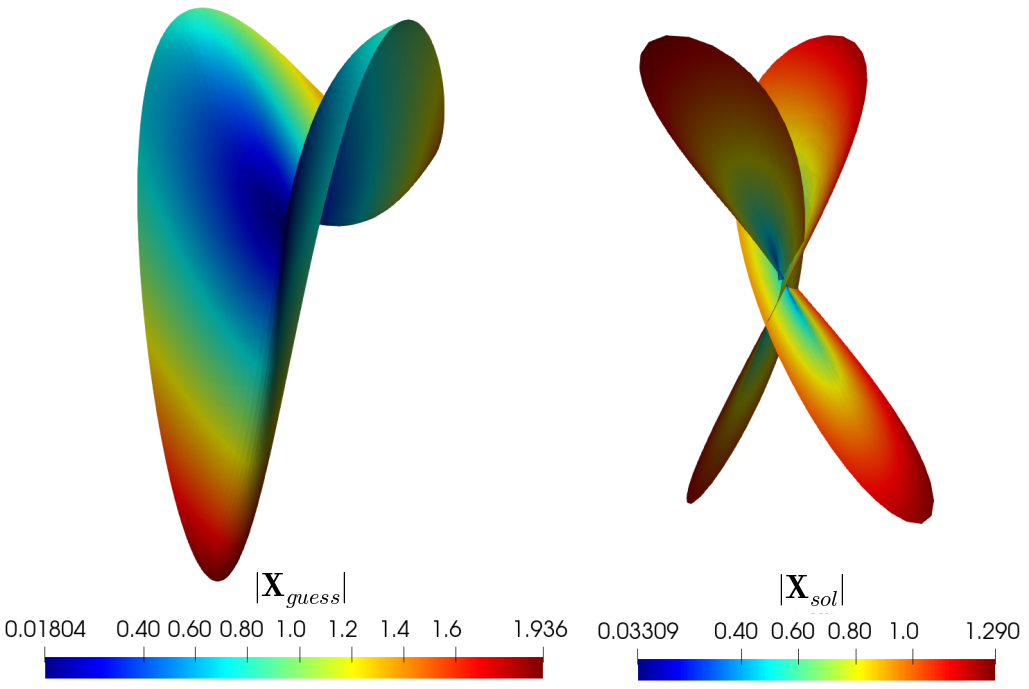}
\caption{Magnitude of the initial guess $\vect{X}_{guess}=\left(\rho\cos\vartheta, - \rho\sin\vartheta, \rho \sin\vartheta+ \sin\left(\frac{\pi}{2}\rho\right)\cos(2\vartheta)\right)$ with $\textgoth{C}_{ijk\ell} = \delta_{ij} \delta_{k\ell} + 2 \left(1 + 10 \sin^2(\rho\cos\vartheta)\right)\delta_{ik}\delta_{jl}$ (left), and of the corresponding obtained solution $\vect{X}_{sol}$ (right), which results in a non-planar solution with self-intersections.}\label{fig:rottura_sym_guess_C_1}
\end{figure}


\bigskip
\bigskip

\section*{Acknowledgements}
F.B. thanks the project {\em Reduced order modelling for numerical simulation of partial differential equations} funded by Università Cattolica del Sacro Cuore.
The work of F.B. has been also partially supported by INdAM$-$GNCS, through the GNCS 2022 project {\em Metodi di riduzione computazionale per le scienze applicate: focus su sistemi complessi}.
The stay of G.B. at Politecnico di Torino was supported by Progetto d’Eccellenza 2018 -- 2022 funded by Ministero dell’Università e della Ricerca (code: E11G18000350001), and the stay of G.B. at Università di Pisa is supported by the European Research Council (ERC), under the European Union's Horizon 2020 research and innovation program, through the project ERC VAREG - {\em Variational approach to the regularity of the free boundaries} (grant agreement No. 853404). The work of G.B. and L.L. has been partially supported by INdAM$-$GNAMPA. The work of A.M. has been also partially supported by INdAM$-$GNFM.

\section*{Declarations}
{\bf Competing interests} The authors declare that they have no known competing financial interests or personal relationships that could have appeared to influence the work reported in this paper.

\bibliographystyle{abbrv}
\bibliography{references}
\end{document}

%% file: defs.tex
\theoremstyle{plain} 

\DeclareMathOperator{\diver}{div}

\DeclareMathOperator*{\argmin}{arg\,min}

\usepackage{mathtools}
\usepackage{pgfplots}
\pgfplotsset{/pgf/number format/use comma,compat=newest}

\renewcommand\epsilon{\varepsilon}
\newcommand{\R}{\mathbb{R}}
\newcommand{\N}{\mathbb{N}}

\newcommand{\vect}[1]{\boldsymbol{#1}}
\newcommand{\tens}[1]{\mathsf{#1}}

\newcommand{\abs}[1]{\left\lvert#1\right\rvert}
\newcommand{\norm}[1]{\left\lVert#1\right\rVert}

\newcommand{\prt}[1]{\left(#1\right)}